\documentclass[12pt]{amsart}
\usepackage{amsmath}
\usepackage{amssymb}
\usepackage{mathtools,nccmath,textcomp}
\usepackage{tikz}
\usepackage[all]{xy}
\usepackage{color,soul}
\usepackage{longtable}
\usepackage[colorlinks=true]{hyperref}
\usepackage{caption}
\usepackage{makecell}
\usepackage{subcaption}
\usepackage{array}
\usepackage{bm}

\allowdisplaybreaks[3]
\newcolumntype{?}{!{\vrule width 1pt}}

\newtheorem{theorem}{Theorem}[section]
\newtheorem{lemma}[theorem]{Lemma}
\newtheorem{proposition}[theorem]{Proposition}

\newtheorem{conjecture}[theorem]{Conjecture}
\newtheorem{algorithm}[theorem]{Algorithm}

\theoremstyle{definition}

\newtheorem{example}[theorem]{Example}

\theoremstyle{remark}
\newtheorem{remark}[theorem]{Remark}

\numberwithin{equation}{section}

\DeclareMathAlphabet{\matheur}{U}{eur}{m}{n}

\newcommand{\ZZ}{\mathbb{Z}}
\newcommand{\NN}{\mathbb{N}}
\newcommand{\QQ}{\mathbb{Q}}
\newcommand{\RR}{\mathbb{R}}

\newcommand{\TT}{\mathbb{T}}

\newcommand{\CC}{\mathbb{C}}

\newcommand{\R}{\mathbb{R}}
\newcommand{\Q}{\mathbb{Q}}
\newcommand{\N}{\mathbb{N}}
\newcommand{\Z}{\mathbb{Z}}

\newcommand{\m}{\mathrm{m}}
\newcommand{\nn}{\mathrm{n}}
\newcommand{\dd}{\;\mathrm{d}}

\newcommand{\re}{\mathop{\mathrm{Re}}} 
\newcommand{\im}{\mathop{\mathrm{Im}}} 

\newcommand{\sgn}{\,\mathrm{sgn}}

\renewcommand{\d}{\mathrm d}

\DeclareMathOperator{\GL}{GL}
\DeclareMathOperator{\SL}{SL}

\DeclareMathOperator{\End}{End}

\DeclareMathOperator{\Gal}{Gal}
\DeclareMathOperator{\Hom}{Hom}
\DeclareMathOperator{\Res}{Res}

\DeclareMathOperator{\Norm}{Norm}

\newmuskip\pFqskip
\mathchardef\pFcomma=\mathcode`, 

\newcommand*\pFq[5]{%
  \begingroup
  \begingroup\lccode`~=`,
    \lowercase{\endgroup\def~}{\pFcomma\mkern\pFqskip}%
  \mathcode`,=\string"8000
  {}_{#1}F_{#2}\Bigl(\genfrac..{0pt}{}{#3}{#4} \,\,\Big| \,\, #5\Bigr)%
  \endgroup
}

\begin{document}

\title[Determinants of Mahler measures and special values of $L$-functions]{Determinants of Mahler measures and special values of $L$-functions}

\author{Detchat Samart}
\address{Department of Mathematics, Faculty of Science, Burapha University, Chonburi 20131, Thailand} 
\email{petesamart@gmail.com}

\author{Zhengyu Tao}
\address{School of Mathematics, Hefei University of Technology, Hefei 230009, People's Republic of China}
\email{taozhy@hfut.edu.cn}



\date{\today}

\begin{abstract}
We consider Mahler measures of two well-studied families of bivariate polynomials, namely $P_t=x+x^{-1}+y+y^{-1}+\sqrt{t}$ and $Q_t=x^3+y^3+1-\sqrt[3]{t}xy$, where $t$ is a complex parameter. In the cases when the zero loci of these polynomials define CM elliptic curves over number fields, we derive general formulas for their Mahler measures in terms of $L$-values of cusp forms. For each family, we also classify all possible values of $t$ in number fields of degree not exceeding $4$ for which the corresponding elliptic curves have complex multiplication. Finally, for all such values of $t$ in totally real number fields of degree $n=2$ and $n=4$, corresponding to elliptic curves $\mathcal{F}_t$ (resp. $\mathcal{C}_t$), we prove that determinants of $n\times n$ matrices whose entries are Mahler measures corresponding to their Galois conjugates are non-zero rational multiples of $L^{(n)}(\mathcal{F}_t,0)$ (resp. $L^{(n)}(\mathcal{C}_t,0)$).
\end{abstract}

\maketitle

\section{Introduction}\label{S:Intro}
For an $n$-variable Laurent polynomial $P\in \CC[x_1^{\pm 1},\ldots, x_n^{\pm 1}]$, the (logarithmic) Mahler measure of $P$ is defined by 
\begin{align*}
\m(P)&= \frac{1}{(2\pi i)^n}\int_{\TT^n}\log |P(x_1,\dots,x_n)|\frac{\d x_1}{x_1}\dotsb\frac{\d x_n}{x_n},
\end{align*}
where $\TT^n=\{(x_1,\ldots,x_n)\in \CC^n : |x_1|=\cdots= |x_n|=1\}$. In other words, $\m(P)$ is the logarithm of the geometric mean of $|P|$ over the $n$-dimensional torus. For $n=1$, this quantity can always be computed explicitly in terms of the zeros of $P$ using Jensen's formula. On the other hand, if $n\ge 2$, $\m(P)$ is notoriously difficult to compute in general. In the case $n=2$, there are certain classes of polynomials whose Mahler measures are known to be related to special values of $L$-functions. Among these are the two families
\begin{align*}
P_t:=P_t(x,y)&=x+\frac{1}{x}+y+\frac{1}{y}+\sqrt{t},\\
Q_t:=Q_t(x,y)&=x^3+y^3+1-\sqrt[3]{t}xy,
\end{align*}
which have been studied extensively by many researchers (see, e.g., \cite{Boyd98,LR07,Rogers11,RZ12,RV99}). The zero loci of $P_t$ and $Q_t$, after desingularization, generically define families of smooth genus-one curves $\mathcal{E}_t$ and $\mathcal{B}_t$, which can be written in the Weierstrass form as follows:
\begin{align}
\mathcal{E}_t: Y^2 & = X^3+(t-8)X^2+16X, \label{E:Et}\\
\mathcal{B}_t: Y^2 & =X^3-27t^{2/3}X^2+216t^{1/3}(t-27)X-432(t-27)^2. \label{E:Bt}
\end{align}
Boyd \cite{Boyd98} originally conjectured based on his numerical computations that, for many large $t=k^2$ (resp. $t=k^3$) with $k\in \ZZ$, the following identities hold:
\begin{align}
\m(P_t) & \stackrel{?}=r_t L'(\mathcal{E}_t,0), \label{E:mPt}\\
\m(Q_t) & \stackrel{?}=s_t L'(\mathcal{B}_t,0),\label{E:mQt}
\end{align}
where $r_t,s_t\in \QQ$. Here the notation $A\stackrel{?}=B$ means that numerical values of $A$ and $B$ coincide at least for the first $25$ decimal digits. Rodriguez Villegas \cite[Table~4]{RV99} then extended Boyd's results by numerically verifying that for several $t\in \ZZ$  
\begin{equation}\label{E:mPt2}
\m(P_t)\stackrel{?}=r_t L'(\mathcal{F}_t,0),
\end{equation}
where $\mathcal{F}_t$ is a quadratic twist of $\mathcal{E}_t$ defined by 
\begin{equation} \label{E:Ft}
\mathcal{F}_t : Y^2=X^3+t(t-8)X^2+16t^2X.
\end{equation}
Although $\mathcal{E}_t$ and $\mathcal{F}_t$ are isomorphic over $\overline{\Q}$, their $L$-functions could be different since they are not always in the same isogeny class. Due to certain $K$-theoretical considerations, which shall be discussed in Section \ref{S:lemmas}, we will focus on the Weierstrass model $\mathcal{F}_t$ rather than $\mathcal{E}_t$ for the rest of this paper. Similarly, following Rodriguez Villegas' observation in \cite{RV99}, we will replace $\mathcal{B}_t$ with the Weierstrass model
\begin{equation}
\mathcal{C}_t: Y^2=X^3-27t^2X^2+216t^3(t-27)X-432t^4(t-27)^2 \label{E:Ct}
\end{equation}
and reformulate \eqref{E:mQt} as
\begin{equation}\label{E:mQt2}
\m(Q_t)\stackrel{?}=s_t L'(\mathcal{C}_t,0),
\end{equation} 
for sufficiently large $t\in \ZZ$. Note that these changes of the underlying curves do not affect the original (conjectural) identities \eqref{E:mPt} and \eqref{E:mQt}.

 To date, only a small number of identities \eqref{E:mPt2} and \eqref{E:mQt2} have been rigorously verified. Comprehensive lists of the proven cases can be found in \cite[Table~1]{Samart21} and \cite[Table~1]{Samart23}. For a survey of methods used in the proofs, the reader is referred to the research monograph \cite{BZ20}. Note that all rational values of $t$ for which $\mathcal{F}_t$ (resp. $\mathcal{C}_t$) has complex multiplication (CM) can be easily examined using the $j$-invariant of $\mathcal{F}_t$ (resp. $\mathcal{C}_t$). In fact, by equating \eqref{E:jFt} to the $13$ rational CM $j$-invariants, it can be seen that there are exactly $3$ rational values of $t$ for which $\mathcal{F}_t$ is a CM curve, namely $t=-16,8,$ and $32$, and the identity \eqref{E:mPt2} is known to be true in these cases (with $r_t=2,1,$ and $1$, respectively), thanks to work of Rodriguez Villegas \cite{RV99} and Rogers \cite{Rogers11}. Similarly, assuming $t\in \QQ$, the curve $\mathcal{C}_t$ has CM precisely when $t=-216,54,$ and $24$. In the first two cases, the identity \eqref{E:mQt2} holds with $s_t=3$ and $\frac32$, respectively \cite{Rogers11,RV99}. The case $t=24$ is slightly different due to a subtle property of the family $Q_t$ and will be discussed exclusively in Section \ref{S:lemmas}.
 
For brevity, we will use the following notation throughout this paper:
\begin{align*}
\mu(t)&:= \m(P_t)=\m\left(x+\frac{1}{x}+y+\frac{1}{y}+\sqrt{t}\right),\\
\nn(t)&:=\m(Q_t)=\m\left(x^3+y^3+1-\sqrt[3]{t}xy\right),\\
\nu(t)&:=\begin{cases}
	\nn(t), & \text{if }\sqrt[3]{t}\notin \mathcal{K}_Q^\circ,\\
	\tilde{\nn}(t), & \text{if }t\in(-1,27),
\end{cases}
\end{align*}
where $\mathcal{K}_Q$ is the set of $k\in\CC$ such that $Q_{k^3}(x,y)$ vanishes on $\TT^2$ and $\tilde{\nn}(t)$ is a modified version of $\nn(t)$, which is defined in \cite{Samart23} (see Section \ref{S:lemmas} below). Since all CM elliptic curves over $\QQ$ in the family $\mathcal{F}_t$ (resp. $\mathcal{C}_t$) have been classified completely and the corresponding Mahler measure (or modified Mahler measure) $\mu(t)$ (resp. $\nu(t)$) is known to be expressible in terms of $L'(\mathcal{F}_t,0)$ (resp. $L'(\mathcal{C}_t,0)$), it is natural to extend these results to an arbitrary number field. In particular, we are interested in the following problems:
\begin{itemize}
\item[(i)] Given a number field $K$, determine all $t\in K$ for which $\mathcal{F}_t$ (resp. $\mathcal{C}_t$) is a CM elliptic curve.
\item[(ii)] For each $t$ in (i), find an explicit relationship between $\mu(t)$ (resp. $\nu(t)$) and a special value of $L(\mathcal{F}_t,s)$ (resp. $L(\mathcal{C}_t,s)$).
\end{itemize}

Problem (ii) is largely motivated by the \textit{Beilinson's conjectures}, a collection of far-reaching conjectures in algebraic $K$-theory relating higher regulators, called the \textit{Beilinson's regulators}, to special $L$-values of algebraic varieties. For precise statements of these conjectures for general algebraic curves over number fields, the reader is referred to \cite{DdJZ06}. Over the past few years, some progress has been made on these problems for real quadratic number fields. For example, Guo, Ji, Liu, and Qin \cite{GJLQ24} give explicit expressions of $\mu(48\pm 32\sqrt{2})$ in terms of regulator integrals and prove that
\begin{equation}\label{E:GJLQ}
\det\begin{pmatrix}
\mu(48+32\sqrt{2}) & 2\mu(48-32\sqrt{2})\\
-2\mu(48-32\sqrt{2}) & 4\mu(48+32\sqrt{2})
\end{pmatrix} = 4L''(\mathcal{F}_{48\pm 32\sqrt{2}},0),
\end{equation}
which is equivalent to a weak version of Beilinson's conjecture for the curve $\mathcal{F}_{48\pm 32\sqrt{2}}$.
In \cite{TGW24}, the second author, Guo, and Wei proposed a systematic approach to find more identities of the form \eqref{E:GJLQ} which correspond to CM elliptic curves defined over real quadratic fields. In particular, they verify analogous identities for $t=272\pm 192\sqrt{2},8\pm 4\sqrt{3},128\pm 64\sqrt{3}, 8\pm 3\sqrt{7},$ and $2048\pm 768\sqrt{7}$. The second author and Guo \cite{TG25} apply a similar method to the family $Q_t$ and prove that
\begin{equation}\label{E:TG}
 \det\begin{pmatrix}
\nu(729+405\sqrt{3}) & \nu(729-405\sqrt{3})\\
\nu(729-405\sqrt{3}) & 4\nu(729+405\sqrt{3})
\end{pmatrix} = \frac{27}{8}L''(\mathcal{C}_{729\pm 405\sqrt{3}},0).
\end{equation}
In this paper, we extend these results by using a modified version of algorithms in \cite{TG25} and \cite{TGW24} to obtain a complete solution to Problem (i) and give illustrative examples  for number fields $K$ with $\deg K\le 4$. Then we tackle Problem (ii) by establishing identities of the form \eqref{E:GJLQ} and \eqref{E:TG} for all totally real algebraic numbers $t$ in (i) of degree at most $4$. For example, we show that for algebraic integers $t_1,t_2,t_3,$ and $t_4$ corresponding to $\#28$ (resp. $\#71$) in Table \ref{quarticGalorbs_table}, the following identities are true for every $t\in \{t_1,t_2,t_3,t_4\}$:
\begin{equation}\label{E:mmquartic1}
	\det\begin{pmatrix}
		\mu(t_1) & \mu(t_2) & \mu(t_3) & \mu(t_4)\\
		\mu(t_2) & -3\mu(t_1) & -3\mu(t_4) & \mu(t_3)\\
		\mu(t_3) & -3\mu(t_4) & -3\mu(t_1) & \mu(t_2)\\
		\mu(t_4) & \mu(t_3) & \mu(t_2) & \mu(t_1)
		\end{pmatrix} = \frac{1}{24}L^{(4)}(\mathcal{F}_t,0),
\end{equation}

\begin{equation}\label{E:mmquartic2}
	\det\begin{pmatrix}
		\nu(t_1) & \nu(t_2) & \nu(t_3) & \nu(t_4)\\
		\nu(t_2) & -5\nu(t_1) & \nu(t_4) & -5\nu(t_3)\\
		\nu(t_3) & \nu(t_4) & 37\nu(t_1) & 37\nu(t_2)\\
		\nu(t_4) & -5\nu(t_3) & 37\nu(t_2) & -185\nu(t_1)
		\end{pmatrix} = \frac{27}{8}L^{(4)}(\mathcal{C}_t,0).
\end{equation}
Note that $t_3$ and $t_4$ for $\#71$ are in $(-1,27)$ while $\sqrt[3]{t_1},\sqrt[3]{t_2}\notin \mathcal{K}_Q^\circ$, so the matrix in \eqref{E:mmquartic2} is a mixture of two (classical) Mahler measures $\nn(t_1),\nn(t_2)$ and two modified Mahler measures $\tilde{\nn}(t_3),\tilde{\nn}(t_4)$. All $4\times 4$ determinant identities for the families $\mathcal{F}_t$ and $\mathcal{C}_t$ are listed in Table~\ref{4x4identities_table} (in forms different from the above examples). To our knowledge, none of these identities is previously known in the literature.

As opposed to \eqref{E:GJLQ} and \eqref{E:TG}, we neither prove these identities using the regulator approach nor do we attempt to translate our results into the $K$-theoretic language (i.e. the Beilinson's conjectures). We focus purely on computational aspects of determinants of Mahler measures and special $L$-values. Our results rely crucially on the fact that, for each algebraic integer $t$ such that $\mathcal{F}_t$ (resp. $\mathcal{C}_t$) has CM, the Mahler measure $\mu(t)$ (resp. $\nu(t)$) can be written as a linear combination of $L$-values of newforms. This can be deduced easily using our general formulas in Proposition \ref{P:MMasLcusp}. After properly assigning coefficients to Mahler measures and aligning them in a $2\times 2$ or $4\times 4$ matrix, the determinant of this matrix becomes the product of those $L$-values (up to an explicit constant) which can then be written as an $L$-value of the corresponding elliptic curve. It should be noted that, given some knowledge about general forms of these matrices and numerical values of the related quantities with sufficient precision, these coefficients can be initially ``guessed'' using the \textsf{PSLQ} algorithm. We also make extensive use of various data from \emph{The L-functions and modular forms database (LMFDB)} \cite{LMFDB}. For example, all newforms appearing in this paper are collected in Table~\ref{newforms_table}, each with an LMFDB label. 

This paper is organized as follows. We give some background on Mahler measure and its connection with the Beilinson's conjectures in Section~\ref{S:lemmas}. This section also contains a brief discussion on the notion of modified Mahler measure (for the family $Q_t$) and some auxiliary lemmas which are needed for the subsequent sections. In Section~\ref{S:CM}, we propose and apply an algorithm for classifying CM elliptic curves in the families $\mathcal{F}_t$ and $\mathcal{C}_t$. Results for the curves defined over number fields of degree not exceeding $4$ will be summarized at the end of the section. Then we establish some general results about factorization of $L$-functions of CM elliptic curves over quadratic and biquadratic fields with certain properties in Section~\ref{S:Lfunc}. In particular, we prove that these elliptic curves are \textit{strongly modular} in the sense of \cite{GQ14}.  The main results of this paper are stated and proven in Section~\ref{S:results}. These include the general formulas for $\mu(t)$ (resp. $\nu(t)$), where $\mathcal{F}_t$ (resp. $\mathcal{C}_t$) is CM, in terms of $L$-values of cusp forms and the determinant formulas for Mahler measures of $\mathcal{F}_t$ (resp. $\mathcal{C}_t$) defined over totally real number fields of degree $2$ and $4$ in terms of $L$-values of the corresponding elliptic curves. Complete lists of related data are tabulated in Appendix~\ref{A:tables}. We give some examples of our results and illustrate the computational process for these examples in Section~\ref{S:examples}. Finally, we give some comments about related results and conjectures in Section~\ref{S:remarks}.

\section{Nuts and bolts}\label{S:lemmas}
For meromorphic functions $f$ and $g$ on a smooth curve $C$, we define the real differential $1$-form $\eta(f,g)$ on $C$ as
\begin{equation*}
\eta(f,g)=\log |f| \dd \arg(g)-\log|g|\dd \arg(f),
\end{equation*}
where $\d \arg(h)= \im\left(\frac{\d h}{h}\right)$. It is easily seen that $\eta(f,g)$ is closed, so its integral over a path $\gamma$ on $C$ only depends on its homotopy class. Let $P(x,y)\in \CC[x,y]$ be irreducible and let $P^*(x)$ be the leading coefficient of $P(x,y)$, seen as a polynomial in $y$. We define the \textit{Deninger path} of the curve $C_P:P(x,y)=0$ by
\[\gamma_P=\{(x,y)\in C_P : |x|=1, |y| \ge 1 \}.\]
Using Jensen's formula, Deninger proved the following result, which relates Mahler measure to the regulator integral.
\begin{theorem}[\cite{BZ20,Deninger97}]\label{T:Deninger}
Suppose $P(x,y)=0$ defines an elliptic curve $E$ and $\gamma_P$ is a finite union of smooth paths in $E$. Then
\begin{equation}\label{E:Deninger}
\m(P)-\m(P^*)=-\frac{1}{2\pi} \int_{\gamma_P}\eta(x,y).
\end{equation}
\end{theorem}

We now state a version of Beilinson's conjecture for elliptic curves (see also \cite[\S 2]{BdLR24} and \cite{DdJZ06}). Let $E$ be an elliptic curve defined over a number field $K$ of degree $n$ and $X=\bigsqcup\limits_{\sigma\in\Hom_{\Q}(K,\CC)}E^\sigma(\CC),$ the disjoint union of $n$ connected Riemann surfaces, one for each embedding of $K$ into $\CC$. Then we have group isomorphisms $H_1(X,\ZZ)\cong\bigoplus_\sigma H_1(E^\sigma(\CC),\Z)\cong\Z^{2n}$. Since complex conjugation acts on $X$ then on $H_1(X,\Z)$ with
\[\overline{\gamma}\in\begin{cases}
	H_1(E^\sigma(\CC),\Z), & \text{if }\sigma\text{ is real,}\\
	H_1(E^{\bar{\sigma}}(\CC),\Z), & \text{if }\sigma\text{ is complex}
\end{cases}\quad \text{for }\gamma\in H_1(E^\sigma(\CC),\Z),\]
the subgroup $H_1(X,\ZZ)^-$ of $H_1(X,\ZZ)$ consisting of the homology classes $\gamma=(\gamma_\sigma)_\sigma$ such that $(\overline{\gamma}_\sigma)_{\bar{\sigma}}=(-\gamma_\sigma)_\sigma$ is a free abelian group of rank $n$. If $E$ has a \textit{tempered} planar model $P(x,y)=0$ (see \cite[\S 8]{RV99} for the definition of a tempered polynomial) such that $x^m$ and $y^n$, considered as functions on $E$, are defined over $K$ for some $m,n\in\Z\setminus\{0\}$, then we have $\{x,y\}\in K_2^T(E)\otimes_\Z\Q$, where $K_2^T(E)$ is the \emph{tame $K_2$ group of $E$}. Although we shall not give a precise definition of $K_2^T(E)$ here, it should be noted that its element can be written as $\sum_i\{f_i,g_i\}$, where $f_i$ and $g_i$ are rational functions in $K(E)$ and $\{\cdot,\cdot\}$ is the Steinberg symbol. By applying $\sigma$ to the coefficients, for each $M=\sum_i\{f_i,g_i\}\in K_2^T(E)$, we have $M^\sigma:=\sum_i\{f_i^\sigma,g_i^\sigma\}\in K_2^T(E^\sigma)$. If we assume further that the Deninger path $\gamma_P$ is closed, then the right-hand side of \eqref{E:Deninger} can be realized as $\langle\gamma_P,\{x,y\}\rangle$, where $\langle \cdot, \cdot \rangle$ is the pairing given by
\begin{equation}\label{E:pairing}
	\begin{aligned}
		H_1(X,\ZZ)^-\times K_2^T(E) &\rightarrow \RR\\
		{\textstyle \left( (\gamma_\sigma)_\sigma, \sum_i\{f_i,g_i\}\right)} & \mapsto \frac{1}{2\pi}\sum_\sigma\sum_i\int_{\gamma_\sigma}\eta(f_i^\sigma,g_i^\sigma).
	\end{aligned}
\end{equation}

Let $K_2^T(E)_\mathrm{int}$ be the \textit{integral tame $K_2$ group of $E$}, which is a subgroup of $K_2^T(E)$ consisting of elements with certain integrality conditions (see \cite[\S 1]{LdJ15} for a more precise definition). Beilinson's conjecture for $E$ can then be stated as follows:
\begin{itemize}
	\item [(i)]  $K_2^T(E)_\mathrm{int}/\text{torsion}$ is a free abelian group of rank $n$,
	\item [(ii)] if $\gamma_1,\dots,\gamma_n$ form a $\Z$-basis of $ H_1(X,\Z)^-$ and $M_1,\dots,M_n$ form a $\Z$-basis of $K_2^T(E)_\mathrm{int}$ up to torsion, then the \emph{Beilinson regulator}
	\begin{equation}\label{eq:nxn_regulator}
		R=|\det(\langle \gamma_i,M_j\rangle)_{1\le i,j \le n}|
	\end{equation}
	is a non-zero rational multiple of $\frac{1}{\pi^{2n}}L(E/K,2)$ (or $L^{(n)}(E/K,0)$ by the (conjectural) functional equation for $L(E/K,s)$).
\end{itemize}
Since it is difficult to compute the group $K_2^T(E)_\mathrm{int}$ in practice, it is certainly unclear how its $\ZZ$-basis can be constructed in order to get a determinant that satisfies the conjecture. In this paper, we only consider the following special setting. 

Let $E$ be an elliptic curve defined over some totally real Galois number field $K$ of degree $n$ such that $E^\sigma$ are $K$-isogenous to each other and let $P(x,y)=0$ be a tempered model for $E$ such that $\{x,y\}\in K_2^T(E)\otimes_\Z\Q$. For each $\sigma\in\Gal(K/\Q)$, fix a $K$-isogeny $\phi_\sigma:E\to E^\sigma$ and a nontrivial element $\gamma_\sigma\in H_1(E^\sigma(\CC),\ZZ)^-\subset H_1(X,\ZZ)^-$. Then each $\phi_\sigma$ induces a pullback map $\phi_\sigma^*: K_2^T(E^\sigma)\to K_2^T(E)$ given by
\[\phi_\sigma^*M=\sum_i\{\phi_\sigma^*f_i,\phi_\sigma^*g_i\},\text{ where } M=\sum_i\{f_i,g_i\}\in K_2^T(E^\sigma)\]
(see \cite{GJLQ24}). Therefore, we can immediately construct $n$ elements
\[\{M_1,\dots,M_n\}=\{\phi_\sigma^*\{x^\sigma,y^\sigma\} : \sigma\in\Gal(K/\Q)\}\subset K_2^T(E)\otimes_\Z\Q.\]
If the Deninger paths for the tempered models $P_\sigma(x,y)=0$ for $E^\sigma$ are all closed and $\gamma_1,\dots,\gamma_n$ are a relabeling of elements in $\{\gamma_\sigma : \sigma\in\Gal(K/\Q)\}$ in some order, then by Theorem \ref{T:Deninger}, each entry in the regulator \eqref{eq:nxn_regulator} should be a rational multiple of $\m(P_\sigma)$ for some $\sigma\in\Gal(K/\Q)$. If we know in addition that the elements $M_i$ are all in $K_2^T(E)_\mathrm{int}\otimes_\Z\Q$ and linearly independent, then by Beilinson's conjecture, it is expected that $R$ is a non-zero rational multiple of $\frac{1}{\pi^{2n}}L(E/K,2)$.

With explicit constructions of the objects mentioned above, results in \cite{GJLQ24}, \cite{TGW24}, and \cite{TG25} are achieved using this approach. However, as mentioned in Section \ref{S:Intro}, we will bypass this regulator-related calculation and directly search for identities between $n\times n$ determinants of Mahler measures and $L^{(n)}(E,0)$. In addition, since every Galois extension of $\QQ$ is either totally real or totally imaginary and Mahler measure is invariant under complex conjugation, we will consider only the totally real cases when searching for such identities.

For $t\in \CC\backslash\{0,16\}$, the curve $P_t(x,y)=0$ can be transformed into the Weierstrass form \eqref{E:Ft} via the birational map
\begin{equation}\label{E:mapFt}
x= \frac{2\sqrt{t}(tX+Y)}{X(X-4t)},\quad y = \frac{2\sqrt{t}(tX-Y)}{X(X-4t)}.
\end{equation}
Similarly, for $t\in \CC\backslash\{0,27\}$, we can transform the Hessian family $Q_t(x,y)=0$ into \eqref{E:Ct} using the map
\begin{equation}\label{E:mapCt}
x= \frac{-18t^{\frac23}X}{3tX+Y-36t^3+972t^2},\quad y = 1- \frac{2Y}{3tX+Y-36t^3+972t^2}.
\end{equation}
It is clear from \eqref{E:mapFt} (resp. \eqref{E:mapCt}) that if we choose $f=x^2$ and $g=y^2$ (resp. $f=x^3$ and $g=y$), then $f$ and $g$ become rational functions on $\mathcal{F}_t$ (resp. $\mathcal{C}_t$) defined over $K=\QQ(t)$. Therefore, by the above arguments, it is possible to relate Mahler measures to suitable regulator integrals appearing in the Beilinson's conjecture. To confirm this assertion, one needs more delicate calculation similar to that in \cite{GJLQ24,TG25,TGW24}, which is beyond the scope of this paper.

Before proving the main results of this section, let us briefly discuss the subtlety in the calculation of $\m(Q_t)$ for certain values of $t$, as mentioned in Section \ref{S:Intro}. Since $Q_t$ is \textit{non-reciprocal}, the set 
\[\mathcal{K}_Q=\{k\in \CC : \{Q_{k^3}=0\}\cap \TT^2 \ne \emptyset\},\]
which is a closed region bounded by the $3$-cusped hypocycloid as illustrated in Figure \ref{thehypocycloid}, has nonempty interior and it can be shown that $\mathcal{K}_Q^\circ\cap \R=(-1,3)$ \cite[\S 2B]{Boyd98}.
\begin{figure}[htbp]
	\centering
	\begin{tikzpicture}[scale=1]
		\draw[->](-2,0)--(3.5,0);
		\draw[->](0,-3)--(0,3);
		\path[fill=gray,opacity=0.3,domain=-pi:pi,smooth]plot({2*cos(\x r)+cos(2*\x r)},{2*sin(\x r)-sin(2*\x r)});
		\draw[domain=-pi:pi,smooth,line width=0.8pt]plot({2*cos(\x r)+cos(2*\x r)},{2*sin(\x r)-sin(2*\x r)});
		\fill (-1,0) circle (1pt);
		\fill (0,0) circle (1pt);
		\fill (3,0) circle (1pt);
		\fill (-1.5,2.598) circle (1pt);
		\fill (-1.5,-2.598) circle (1pt);
		\node[below,font=\small]at(-1.3,0){$-1$};
		\node[below,font=\small]at(0.2,0){$0$};
		\node[below,font=\small]at(3,0){$3$};
		\node[left,font=\small]at(-1.45,2.7){$3\omega$};
		\node[left,font=\small]at(-1.45,-2.7){$3\omega^2$};
	\end{tikzpicture}
	\caption{The closed region $\mathcal{K}_Q$ ($\omega=e^{\frac{2\pi i}{3}}$)}
	\label{thehypocycloid}
\end{figure}
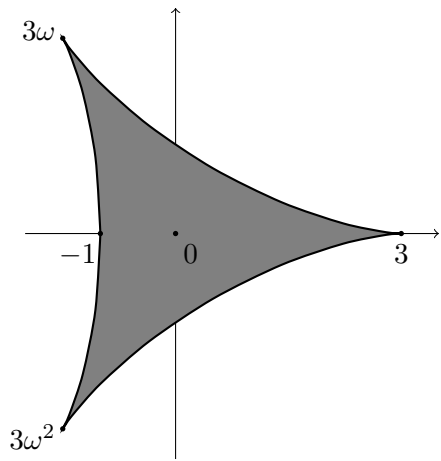
 If $k\in \mathcal{K}_Q^\circ$ (e.g. $k=\sqrt[3]{24}$), then the curve $\{Q_{k^3}=0\}$ intersects the torus $\mathbb{T}^2$ nontrivially and the conjecture of the form \eqref{E:mQt2} becomes invalid because its Deninger path is not closed. The Mahler measure of $Q_t$ for $t\in (-1,27)$ is investigated in recent work of the first author \cite{Samart23}, which shall be described below. Consider the family
\[\tilde{Q}_t:=\tilde{Q}_t(x,y)= y^2+(x^2-\sqrt[3]{t}x)y+x,\]
which is closely related to $Q_t$ since the zero locus of $\tilde{Q}_t$ is isomorphic and isogenous to $\mathcal{C}_t$ over $\QQ(t)$. Moreover, by a simple change of variables, we have that $\m(\tilde{Q}_t)=\m(Q_t)$. Observe that $\tilde{Q}_t$ is written as a quadratic polynomial in variable $y$, so its Mahler measure is much easier to compute than $\m(Q_t)$. Using the quadratic formula, we can factorize $\tilde{Q}_t$ as 
\[\tilde{Q}_t = (y-y_+(x))(y-y_-(x)),\]
where 
\[y_{\pm}(x)= -(x^2-\sqrt[3]{t}x)\left(\frac12\pm \sqrt{\frac14-\frac{1}{x(x-\sqrt[3]{t})^2}}\right).\]
For $t\in (-1,27)$, define 
\[I(t)=\frac{1}{2\pi} \int_{-c(t)}^{c(t)}\log|y_+(e^{i\theta})|\d \theta,\quad J(t)=\frac{1}{2\pi} \int_{c(t)}^{2\pi -c(t)}\log|y_+(e^{i\theta})|\d \theta,\]
where $c(t)=\cos^{-1}\left(\frac{\sqrt[3]{t}-1}{2}\right).$ By \cite[Lemma~3]{Samart23} and Jensen's formula, we have that $\m(\tilde{Q}_t)=I(t)+J(t)$. To close the Deninger path, the first author considered
\[\tilde{\nn}(t):=I(t)-2J(t),\]
which is referred to as the \emph{modified Mahler measure} of $\tilde{Q}_t$, and discovered from his computational experiments that for every non-zero integer $t\in (-1,27)$, 
\begin{equation*}
\tilde{\nn}(t)\stackrel{?}= s_tL'(\mathcal{C}_t,0),
\end{equation*}
where $s_t\in \QQ$ (see \cite[Table~2]{Samart23}). This identity for the case $t=24$, which corresponds to a CM elliptic curve $\mathcal{C}_{t}$ of conductor $27$ and $s_t=-3$, has recently been proven by the second author and Guo \cite[Theorem~1.5]{TG25}. It should be remarked that, if $t$ is an algebraic integer, cube roots of the Galois conjugates of $t$ can possibly scatter in both $\mathcal{K}_Q^\circ$ and its complement. In this case, the corresponding determinant formulas obtained in Section~\ref{S:results} involve both Mahler measures and modified Mahler measures (e.g. \eqref{E:mmquartic2}). For conciseness, we use $\nu(t)$ as defined in Section~\ref{S:Intro} to unify $\nn(t)$ and $\tilde{\nn}(t)$ in our results. Note further that the Galois conjugates of $t$ whose cube roots belong to the region $\mathcal{K}_Q^\circ$ (in fact the interval $(-1,3)$) are put in boldface in Table~\ref{2x2identities_table} and Table~\ref{quarticGalorbs_table}.

To apply our classification algorithm in Section \ref{S:CM}, we need some auxiliary lemmas which give explicit correspondences between period ratios of the curves $\mathcal{F}_t$ and $\mathcal{C}_t$ and the parameter $t$ via certain modular functions. Recall that the modular lambda function is a holomorphic function on the upper half-plane $\mathcal{H}:=\{\tau\in \CC : \im \tau>0\}$ defined by 
\begin{equation*}
\lambda(\tau)=16\left(\frac{\eta\left(\frac{\tau}{2}\right)\eta^2(2\tau)}{\eta^3(\tau)}\right)^8,
\end{equation*}
where $\eta(\tau)$ is the Dedekind eta function. Denote 
\begin{equation*}
t_P(\tau)=\frac{16}{\lambda(2\tau)}=\frac{\eta(2\tau)^{24}}{\eta(\tau)^8\eta(4\tau)^{16}},\quad t_Q(\tau)=27+\left(\frac{\eta(\tau)}{\eta(3\tau)}\right)^{12},
\end{equation*}
which are Hauptmoduls for $\Gamma_0(4)$ and $\Gamma_0(3)$, respectively.
\begin{lemma}\label{L:tP}
	Let $t\in \CC\backslash\{0,16\}$ and 
	\[\tau'= \frac{i}{2}\frac{\pFq{2}{1}{\frac12, \frac12}{1}{1-\frac{16}{t}}}{\pFq{2}{1}{\frac12, \frac12}{1}{\frac{16}{t}}} .\] Then we have
	\begin{itemize}
		\item [(i)]  $\mathcal{F}_t \cong \CC/ (\Z+\Z(4\tau'))$ and
		\item [(ii)] $t=t_P(\tau')$.
	\end{itemize}
\end{lemma}
\begin{proof}
	To prove (i), it suffices to show that $j(\mathcal{F}_t)=j(4\tau').$ Using a standard computer algebra system such as \textsc{Maple} or \textsc{SageMath}, one sees readily that 
	\begin{equation}\label{E:jFt}
		j(\mathcal{F}_t)=\frac{(t^2-16t+16)^3}{t(t-16)}.
	\end{equation}
	Next, we will compute $j(4\tau')$. It is a well-known fact that for any $\tau\in \mathcal{H}$
	\begin{equation}\label{E:j}
		j(\tau)=\left(\frac{\mathfrak{f}^{24}(\tau)-16}{\mathfrak{f}^{8}(\tau)}\right)^3,
	\end{equation}
	where $\mathfrak{f}(\tau)=e^{-\frac{\pi i}{24}}\frac{\eta\left(\frac{\tau+1}{2}\right)}{\eta(\tau)}$ (see \cite[\S 1]{YZ97}). By Ramanujan's theory of elliptic functions \cite[Entry~12(v)]{Berndt91}, if $s_2(q)= \mathfrak{f}^{24}(2\tau),$ where $q=e^{2\pi i \tau}$, and \[q_2(\alpha)=\exp\left(-\pi \frac{\pFq{2}{1}{\frac12, \frac12}{1}{1-\alpha}}{\pFq{2}{1}{\frac12, \frac12}{1}{\alpha}} \right),\]
	then $s_2(q_2(\alpha))=\frac{16}{\alpha(1-\alpha)}.$ Hence
	\begin{equation*}
		\mathfrak{f}^{24}(2\tau')=s_2\left(q_2\left(\frac{16}{t}\right)\right)=\frac{t^2}{t-16}.
	\end{equation*}
	By \eqref{E:j}, it follows that 
	\[j(2\tau')=\frac{(t^2-16t+256)^3}{t^2(t-16)^2}.\]
	Since $X_0(2)$ has genus zero, there exists a modular function $f(\tau)$ such that $\Q(j(\tau),j(2\tau))=\Q(f(\tau))$. Moreover, such $f(\tau)$ can be chosen so that
	\[j(\tau)= \frac{(f(\tau)+256)^3}{f(\tau)^2},\quad j(2\tau)= \frac{(f(\tau)+16)^3}{f(\tau)}\]
	(see \cite[P2]{vH15}). Hence we can deduce that 
	\[j(4\tau')=\frac{(t^2-16t+16)^3}{t(t-16)}=j(\mathcal{F}_t).\]
	The identity (ii) follows directly from the inverse property of the lambda function \cite[b7174d]{Lambda}.
\end{proof}

\begin{lemma}\label{L:tQ}
	Let $t\in \CC\backslash\{0,27\}$ and
	\[\tau'= \frac{i}{\sqrt{3}} \frac{\pFq{2}{1}{\frac13, \frac23}{1}{1-\frac{27}{t}}}{\pFq{2}{1}{\frac13, \frac23}{1}{\frac{27}{t}}}.\] Then we have 
	\begin{itemize}
		\item [(i)]  $\mathcal{C}_t \cong \CC/(\Z+\Z\tau')$ and
		\item [(ii)] $t=t_Q(\tau')$.
	\end{itemize}
\end{lemma}
\begin{proof}
	The isomorphism (i) for $t\in (27,\infty)$ is proven in \cite[Lemma~3.4]{Samart16} (in a slightly different form under $\tau'\rightarrow -\frac{1}{\tau'}$). Then we can easily extend this result to $t\in \CC\backslash\{0,27\}$ using the fact that both $j(\mathcal{C}_t)$ and $j(\tau')$ (as a function of $t$) are meromorphic functions. The identity (ii) is also an immediate consequence of Ramanujan's theory of elliptic functions in signature $3$ \cite[\S 33]{Berndt98}. More precisely, if $s_3(q)=\left(27\left(\frac{\eta(3\tau)}{\eta(\tau)}\right)^6+\left(\frac{\eta(\tau)}{\eta(3\tau)}\right)^6\right)^2$ and \[q_3(\alpha)=\exp\left(-\frac{2\pi}{\sqrt{3}}\frac{\pFq{2}{1}{\frac13, \frac23}{1}{1-\alpha}}{\pFq{2}{1}{\frac13, \frac23}{1}{\alpha}} \right),\]
	then $s_3(q_3(\alpha))=\frac{27}{\alpha(1-\alpha)}.$
\end{proof}

\begin{lemma}\label{L:tPtQ}
For $\tau\in \mathcal{H}$, we have
\[j(\tau)=\frac{(t_P(\tau)^2+224t_P(\tau)+256)^3}{t_P(\tau)(t_P(\tau)-16)^4}=\frac{t_Q(\tau)(t_Q(\tau)+216)^3}{(t_Q(\tau)-27)^3}.\]
\end{lemma}
\begin{proof}
The first equality follows from the well-known identity 
\[j(\tau)=\frac{256(1-\lambda+\lambda^2)^3}{\lambda^2(1-\lambda)^2}\]
and the modular equation 
\[\lambda(\tau)=\frac{4\sqrt{\lambda(2\tau)}}{\bigl(1+\sqrt{\lambda(2\tau)}\bigr)^2}\]
(see \cite[Theorem~4.4, Eq. (4.1.3)]{BB87}).
For the second equality, see the proof of \cite[Theorem~2.9]{Samart15}.
\end{proof}

\section{Algorithm for CM curve classification} \label{S:CM}
In this section, we propose an algorithm which can be used to find an exhaustive list of CM elliptic curves in the family $\mathcal{F}_t$ (resp. $\mathcal{C}_t$), up to any given upper bound for the algebraic degree of $t$. Let us first recall the following results of Rodriguez Villegas \cite[\S 14]{RV99} regarding formulas for Mahler measures $\mu(t)$ and $\nu(t)$, suitably parametrized by $t_P(\tau)$ and $t_Q(\tau)$, respectively,  in terms of Eisenstein-Kronecker series.
\begin{theorem}[\cite{RV99}]\label{RVformula}
	Let $F_P\subset\mathcal{H}$ be the fundamental domain of the congruence subgroup $\Gamma_0(4)$ formed by the geodesic triangle of vertices $ i\infty,0,1/2$ and its reflection along the imaginary axis (see Figure \ref{fig:F_PandF_Q} (a)). Also, let $F_Q\subset\mathcal{H}$ be the fundamental domain of $\Gamma_0(3)$ formed by the geodesic triangle of vertices $i\infty,0,\tfrac{1+i/\sqrt{3}}{2}$ and its reflection along the imaginary axis (see Figure \ref{fig:F_PandF_Q} (b)). 
	\begin{itemize}
		\item[(i)] For any $\tau_0\in F_P$, we have
		\begin{equation}\label{eq:RVformu}
			\mu(t_P(\tau_0))=\frac{16\im(\tau_0)}{\pi^2}\underset{m,n\in\Z}{{\sum}'}\frac{\chi_{-4}(n)(4m\re(\tau_0)+n)}{\left|4m\tau_0+n\right|^4}.
		\end{equation}
		\item[(ii)] For any $\tau_0\in F_Q$, we have
		\begin{equation}\label{eq:RVfornu}
			\nu(t_Q(\tau_0))=\begin{cases}
				\frac{27\sqrt{3}\im(\tau_0)}{4\pi^2}\underset{m,n\in\Z}{{\sum}'}\frac{\chi_{-3}(n)(3m\re(\tau_0)+n)}{\left|3m\tau_0+n\right|^4}, & \text{if }\sqrt[3]{t_Q(\tau_0)}\notin\mathcal{K}^\circ_Q,\\
				\frac{27\sqrt{3}\im(\tau_0)}{(1-3\sgn(t_Q(\tau_0)))\pi^2}\underset{m,n\in\Z}{{\sum}'}\frac{\chi_{-3}(n)(3m\re(\tau_0)+n)}{\left|3m\tau_0+n\right|^4}, & \text{if }t_Q(\tau_0)\in(-1,27),
			\end{cases}
		\end{equation}
	\end{itemize}
	where $\chi_{-f}=\left(\frac{-f}{\cdot}\right)$ and $\underset{m,n\in\Z}{\sum'}$ means that $(m,n)=(0,0)$ is excluded from the summation.
\end{theorem}

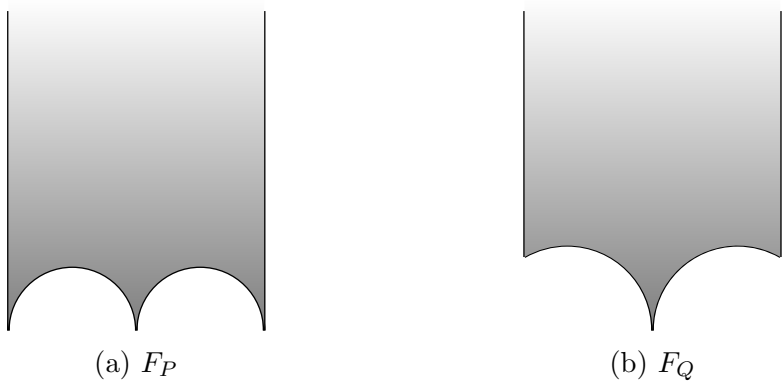
\begin{figure}[htbp]
	\centering
	\begin{subfigure}{.45\textwidth}
		\centering
		\begin{tikzpicture}[scale=1.7]
			\draw[line width=1pt] (-1,0) -- (-1,2.5);
			\draw[line width=1pt] (1,0) -- (1,2.5);
			\draw[line width=1pt] (0,0) arc(0:180:0.5);
			\draw[line width=1pt] (0,0) arc(180:0:0.5);
			\shade[bottom color=gray,top color=white,opacity=0.35] (-1,0) -- (-1,2.59) -- (1,2.59) -- (1,0) arc(0:180:0.5) arc(0:180:0.5);
		\end{tikzpicture}
		\caption*{(a) $F_P$}
	\end{subfigure}\hspace{-20pt}
	\begin{subfigure}{.45\textwidth}
		\centering
		\begin{tikzpicture}[scale=1.7]
			\draw[line width=1pt] (-1,0.575) -- (-1,2.5);
			\draw[line width=1pt] (1,0.575) -- (1,2.5);
			\draw[line width=1pt] (0,0) arc(0:120:0.6666);
			\draw[line width=1pt] (0,0) arc(180:60:0.6666);
			\shade[bottom color=gray,top color=white,opacity=0.35] (-1,0.575) -- (-1,2.6) -- (1,2.6) -- (1,0.575) arc (60:180:0.6666) arc (0:120:0.6666);
		\end{tikzpicture}
		\caption*{(b) $F_Q$}
	\end{subfigure}
	\caption{Fundamental domains of $\Gamma_0(4)$ and $\Gamma_0(3)$}
	\label{fig:F_PandF_Q}
\end{figure}

Here we have included the additional formula for the case $t\in (-1,27)$ in \eqref{eq:RVfornu}, which is proven by the second author and Guo \cite[Eq.~(4.5)]{TG25}.  Note that for every CM point $\tau_0\in\mathcal{H}$, there must exist three unique integers $a(\tau_0),b(\tau_0),c(\tau_0)$ with $a>0$ and $\gcd(a,b,c)=1$ such that $a\tau_0^2+b\tau_0+c=0$. Thus, we can denote $\tau_0$ as $[a,b,c]=-\frac{b}{2a}+i\frac{\sqrt{4ac-b^2}}{2a}$. Based on Lemma \ref{L:tPtQ}, we can use the following algorithm, which is slightly adjusted from those in \cite{TGW24} and \cite{TG25}, to classify CM elliptic curves in the families $\mathcal{F}_t$ and $\mathcal{C}_t$ in accordance with the degree of the base field $\Q(t)$. To be precise, while the algorithms in \cite{TGW24} and \cite{TG25} rely on certain upper bounds of class numbers, we remove such a condition completely in order to obtain the desired CM points in the fundamental domains of $\Gamma_0(4)$ and $\Gamma_0(3)$. The completeness of the list of such CM points is guaranteed by Proposition~\ref{P:CMPQ}. In addition, our proposed algorithm directly computes values of the Hauptmoduls $t_P$ and $t_Q$ at CM points using modular equations, whereas the previous two works did not incorporate the computation of these Hauptmoduls into their algorithms.

\begin{algorithm}\label{A:CM}
	Let $n\in \NN$ and let $D\equiv 0$ or $1\ (\mathrm{mod}\ 4)$ be a negative integer with $h(D)\le n$, where $h(D)$ denotes the number of inequivalent primitive binary quadratic forms of discriminant $D$. Perform the following operations: 
	\begin{itemize}
		\item [(i)] Solve the system
		\[\begin{cases}
			a,b,c\in\Z,a>0,\\
			\gcd(a,b,c)=1,\\
			b^2-4ac=D, \\
			|b|\le a\le c
		\end{cases}\]
		to determine a full list of CM points $[a,b,c]$ in the fundamental domain $F=\{\tau\in\mathcal{H} : -1/2\le \re(\tau)\le 1/2,|\tau|\ge 1\}$ of $\SL_2(\Z)$ with discriminant $D$.
		\item [(ii)] Use the right coset representatives of $\Gamma_0(4)$ (resp. $\Gamma_0(3)$) in $\SL_2(\Z)$ to convert the points obtained in \emph{(i)} to determine a full list of CM points in $F_P$ (resp. $F_Q$) with discriminant $D$.
		\item [(iii)] For each CM point $\tau_0$ in $F_P$ (resp. $F_Q$) with discriminant $D$ obtained in \emph{(ii)}. Calculate the exact value of $j(\tau_0)$ by using CM theory and then calculate the exact value of $t_P(\tau_0)$ (resp. $t_Q(\tau_0)$) by using Lemma \ref{L:tPtQ}. If $t_P(\tau_0)\neq 0,16$ (resp. $t_Q(\tau_0)\neq 0, 27$) and the degree of $t_P(\tau_0)$ (resp. $t_Q(\tau_0)$) does not exceed $n$, then output $\tau_0$.
	\end{itemize}
\end{algorithm}
\begin{remark}\label{R:algo}
In Algorithm~\ref{A:CM} (iii), by ``using CM theory'' we refer to the fact that the degree of the algebraic integer $j(\tau_0)$ coincides with $h(D)$. We first compute $j(\tau_0)$ numerically and use \texttt{RootApproximant} in \textsc{Mathematica} to identify its minimal polynomial with prescribed degree. Since the minimal polynomial has integer coefficients, its correctness can be easily verified using numerical values of the Galois conjugates of $j(\tau_0)$.
\end{remark}

\begin{proposition}\label{P:CMPQ}
	Let $n\in\NN$. If we apply Algorithm \ref{A:CM} to \emph{all} negative integer $D\equiv 0$ or $1\ (\mathrm{mod}\ 4)$ with $h(D)\le n$ and denote the set of the output CM points as $\mathsf{CM}_{P,n}$ (resp. $\mathsf{CM}_{Q,n}$), then we have
	\[\mathsf{T}_{P,n}:=\{t\in\CC : \mathcal{F}_t\text{ has CM with }[\Q(t):\Q]\le n\}=\{t_P(\tau_0) : \tau_0\in\mathsf{CM}_{P,n}\}.\]

	\[\left(\text{resp. }\mathsf{T}_{Q,n}:=\{t\in\CC : \mathcal{C}_t\text{ has CM with }[\Q(t):\Q]\le n\}=\{t_Q(\tau_0):\tau_0\in\mathsf{CM}_{Q,n}\}.\right)\]
	In other words, we can use Algorithm \ref{A:CM} to find \emph{all} CM elliptic curves in the family $\mathcal{F}_t$ (resp. $\mathcal{C}_t$) defined over number fields with degree not exceeding $n$.
\end{proposition}

\begin{proof}
	Let $\mathcal{F}_t$ (resp. $\mathcal{C}_t$) be a CM elliptic curve with $[\Q(t):\Q]\le n$. Then we have $t\neq 0,16$ (resp. $t\neq 0, 27$) since otherwise $\mathcal{F}_t$ (resp. $\mathcal{C}_t$) will not be an elliptic curve. By Lemma \ref{L:tP} (resp. Lemma \ref{L:tQ}), there exists a CM point $\tau'\in\mathcal{H}$ such that $t=t_P(\tau')$ (resp. $t=t_Q(\tau')$). Since $t_P$ (resp. $t_Q$) is the Hauptmodul for $\Gamma_0(4)$ (resp. $\Gamma_0(3)$), we know that there also exists a CM point $\tau_0$ in the fundamental domain $F_P$ (resp. $F_Q$) such that $t=t_P(\tau_0)$ (resp. $t_Q(\tau_0)$). Let $D$ be the discriminant of $\tau_0$. By Lemma \ref{L:tPtQ}, we have $j(\tau_0)\in\Q(t_P(\tau_0))$ (resp. $\Q(t_Q(\tau_0))$), implying
	\[h(D)=[\Q(j(\tau_0)):\Q]\le[\Q(t_P(\tau_0)):\Q]\le n\]

	\[\left(\text{resp. }h(D)=[\Q(j(\tau_0)):\Q]\le[\Q(t_Q(\tau_0)):\Q]\le n\right).\]
	Hence we have $\tau_0\in\mathsf{CM}_{P,n}$ (resp. $\mathsf{CM}_{Q,n}$).

	Conversely, let $\tau_0\in\mathsf{CM}_{P,n}$ (resp. $\mathsf{CM}_{Q,n}$). According to the last step of Algorithm \ref{A:CM}, we know that $t=t_P(\tau_0)$ (resp. $t_Q(\tau_0)$) is an algebraic number with degree $\le n$ and $\mathcal{F}_t$ (resp. $\mathcal{C}_t$) is non-singular. Therefore, it remains to show that $\mathcal{F}_t$ (resp. $\mathcal{C}_t$) has CM. By Lemma~\ref{L:tP} (resp. Lemma~\ref{L:tQ}), there exists $\tau'\in\mathcal{H}$ such that
\begin{equation*}
		\mathcal{F}_t\cong \CC/ (\Z+\Z(4\tau'))\quad\text{ and }\quad t=t_P(\tau')
	\end{equation*}

	\begin{equation*}
		(\text{resp. }\mathcal{C}_t\cong \CC/(\Z+\Z\tau')\quad\text{ and }\quad t=t_Q(\tau')).
	\end{equation*}
Moreover, since $t_P$ (resp. $t_Q$) is a Hauptmodul, $\tau_0$ and $\tau'$ must be in the same $\Gamma_0(4)$ (resp. $\Gamma_0(3)$) orbit. In particular, $\tau'$ is also a CM point and $\mathcal{F}_t$ (resp. $\mathcal{C}_t$) has CM, as desired.
\end{proof}

We implement Algorithm \ref{A:CM} in \textsc{Mathematica} and apply it to all negative discriminants $D$ with $h(D)\le 4$.  Comprehensive lists of $D$ for which $h(D)=1,2$ can be found in \cite[Proposition~3.1]{TGW24} and the remaining cases can be done using the same arguments. Indeed, the negative fundamental discriminants $D$ satisfying $h(D)=3$ (resp. $h(D)=4$) are listed in \cite{OEIS_A006203} (resp. \cite{OEIS_A013658}). There has been extensive work on computing class numbers of imaginary quadratic fields, which underlies such lists (see, e.g., \cite[\S 5.3]{Cohen93}). Then one can apply  \cite[Corollary~7.28]{Cox13} to determine the remaining (non-fundamental) values of $D$ with the prescribed class number. It turns out that $\mathsf{T}_{P,4}$ consists of
\begin{itemize}
	\item $3$ rational integers $t$;
	\item $12$ pairs of real quadratic integers $t,t^\sigma$ with $\Gal(\Q(t)/\Q)=\langle\sigma \rangle$;
	\item $4$ pairs of imaginary quadratic integers $t,\bar{t}$;
	\item $24$ Galois orbits of totally real quartic integers $t_1,t_2,t_3,t_4$;
	\item $24$ conjugate orbits of non-totally real quartic integers $t_1,t_2,t_3,t_4$,
\end{itemize}
and $\mathsf{T}_{Q,4}$ consists of
\begin{itemize}
    \item $3$ rational integers $t$;
	\item $15$ pairs of real quadratic integers $t,t^\sigma$ with $\Gal(\Q(t)/\Q)=\langle\sigma \rangle$;
	\item $2$ pairs of imaginary quadratic integers $t,\bar{t}$;
	\item $4$ conjugate orbits of non-totally real cubic integers $t_1,t_2,t_3$;
	\item $24$ Galois orbits of totally real quartic integers $t_1,t_2,t_3,t_4$;
	\item $25$ conjugate orbits of non-totally real quartic integers $t_1,t_2,t_3,t_4$.
\end{itemize}
The rational integers for both families are exactly the ones mentioned in Section \ref{S:Intro}. The real quadratic integers for both families are collected in Table~\ref{2x2identities_table}. It is also worth mentioning that all of the totally real quartic integers are of biquadratic type, as displayed in Table~\ref{quarticGalorbs_table}. All the data above are collected in a separate file, which can be accessed at \cite{ST24}.

\section{$L$-functions of CM elliptic curves}\label{S:Lfunc}
It is a well-known fact due to a classical result of Deuring that the $L$-series of a CM elliptic curve defined over an arbitrary number field is expressible in terms of the Hecke $L$-series attached to Gr\"{o}ssencharacters, so that it can be analytically continued to the entire complex plane. More precisely, we have the following result.
\begin{proposition}
Let $E$ be a CM elliptic curve defined over a number field $K$ and define the complete $L$-function of $E/K$ as
\[\Lambda(E/K,s):=(\Norm(\mathcal{N}_K(E))d_K^2)^{s/2}((2\pi)^{-s}\Gamma(s))^{[K:\Q]}L(E/K,s),\]
where $\mathcal{N}_K(E)$ denotes the conductor of $E/K$ and $d_K$ denotes the discriminant of $K$. Then $\Lambda(E/K,s)$ admits an analytic continuation to the entire complex plane and satisfies the functional equation
\begin{equation}\label{E:FE}
\Lambda(E/K,s)=w_{E/K}\Lambda(E/K,2-s),
\end{equation}
where $w_{E/K}\in \{\pm 1\}$ is the root number of $E/K$.
\end{proposition}

 We shall examine structures of $L$-functions of CM elliptic curves defined over a quadratic or quartic number field with certain properties, including those appearing in Section \ref{S:results}. The following proposition will be crucial when linking products of $L$-values of cusp forms to $L$-values of elliptic curves in the proofs of our main results. Recall that a newform $f=\sum a_nq^n \in \mathcal{S}_k(\Gamma_1(N))$ is said to have \textit{complex multiplication} (CM) by a Dirichlet character $\phi$ of conductor $M$ associated to an imaginary quadratic field if $a_p=\phi(p)a_p$ for all prime $p\nmid NM$. 

For an elliptic curve $E$ defined over a number field $K$ and a subfield $L$ of $K$, we let $\Res_{K/L}(E)$ denote the Weil restriction of $E$, which is an abelian variety of dimension $[K:L]$ over $L$. Also, we denote the conductor of an abelian variety $A$ over $L$ by $\mathcal{N}_L(A)$. 

\begin{proposition}\label{P:Lsplit}
Let $K$ be a totally real number field and let $E$ be a CM elliptic curve over $K$ which is $K$-isogenous to all of its Galois conjugates.
\begin{itemize}
\item [(i)] Suppose $K$ is a quadratic field and $A=\Res_{K/\Q}(E)$. If $A$ is simple over $\Q$, then $L(E,s)=L(f,s)L(f^\sigma,s)$ for some CM newform $f\in \mathcal{S}_2(\Gamma_1(N);M)$, where $M=\End_\QQ(A)\otimes \QQ$ is a quadratic number field with $\Gal(M/\Q)=\langle \sigma \rangle$, and $\mathcal{N}_\QQ(A)=N^2$.  Otherwise, $L(E,s)=L(f_1,s)L(f_2,s)$ for some CM newforms $f_i\in \mathcal{S}_2(\Gamma_0(N_i);\Q)$ and $\mathcal{N}_\QQ(A)=N_1N_2$.
\item [(ii)] Suppose $K$ is a biquadratic field and $A=\Res_{K/\Q}(E)$. Then there exist CM newforms $f_i\in \mathcal{S}_2(\Gamma_1(N_i))$, $1\le i\le 4$, such that 
$L(E,s)=L(f_1,s)L(f_2,s)L(f_3,s)L(f_4,s)$ and $\mathcal{N}_\QQ(A)=\prod_{i=1}^4 N_i$.
\end{itemize}
\end{proposition}
\begin{proof} 
To prove (i), we follow the arguments in the proof of \cite[Lemma 4.2]{GJLQ24}. First, let $K=\QQ(\sqrt{d})$ be a quadratic field with $E'$ being the Galois conjugate of $E$ and assume that $A=\Res_{K/\QQ}(E)$ is simple over $\QQ$. Then it follows that $\End_\QQ^0(A):=\End_\QQ(A)\otimes \QQ$ is a division algebra.
By the universal property of the restriction of scalars, we have the following $\QQ$-vector space isomorphism
\[\End_\QQ^0(A)\cong \prod_{\tau\in \Gal(K/\QQ)}\Hom_K^0(E^\tau,E),\]
where $\Hom_K^0(A,B)=\Hom_K(A,B)\otimes \QQ.$ Since $K$ is totally real and $E$ has CM, we have $\End_K(E)\cong \Z$. Moreover, since $E$ is $K$-isogenous to $E'$, we have that $\Hom_K(E,E')$ is a free $\ZZ$-module of the same rank as $\End_K(E)$, implying that $\End_\QQ^0(A)$ has dimension $2$ over $\QQ$. Since every $2$-dimensional division algebra is commutative, $\End_\QQ^0(A)$ must be a quadratic number field. In other words, the abelian surface $A$ is of $\GL_2$-type, as defined in \cite{Ribet04}. Then by the Serre's modularity conjecture, which has now been proven, thanks to work of Ribet \cite{Ribet04} and Khare and Wintenberger \cite{KW09-I,KW09-II}, we have that $A$ is isogenous over $\QQ$ to the modular abelian surface $A_f$ attached to some newform $f\in \mathcal{S}_2(\Gamma_1(N))$ via Shimura's construction. We have in addition that $f$ has CM and its coefficient field $M$ is isomorphic to $\End_\QQ(A_f)\otimes \QQ$, which is a quadratic number field in this case (see, e.g., \cite{Gonzalez11}). Hence by the defining property of the restriction of scalars it follows that
\[L(E/K,s)=L(A/\QQ,s)=L(A_f,s)=L(f,s)L(f^\sigma,s),\]
where $\sigma$ is the generator of the Galois group $\Gal(M/\QQ).$  Moreover, by Carayol's theorem \cite{Carayol89}, we have $\mathcal{N}_\QQ(A)=\mathcal{N}_\QQ(A_f)=N^2$. On the other hand, if $A$ is not simple, then $A \sim_{\QQ} E_1 \times E_2$ for some elliptic curves $E_1$ and $E_2$ over $\QQ$, implying 
$$
L(E / K, s)=L(A / \QQ, s)=L\left(E_1, s\right) L\left(E_2, s\right).
$$
 It then follows immediately from the modularity theorem that the right-hand side is the product of $L$-functions of some CM newforms $f_i\in \mathcal{S}_2(\Gamma_0(N_i);\QQ)$, $i=1,2$, corresponding to $E_1$ and $E_2$, and $\mathcal{N}_\QQ(A)=\mathcal{N}_\QQ(E_1)\mathcal{N}_\QQ(E_2)=N_1N_2.$

The statement (ii) can be proven in a similar manner. First, assume that there exist a proper subfield $L$ of $K$ for which $E$ is isogenous to a base change of an elliptic curve $F$ over $L$. Hence $\Res_{K/L}(E)$ is not simple (over $L$).  If $L=\QQ$, then we have 
\[L(E/K,s)=\prod_{\chi\in \widehat{G}}L(F\otimes \chi,s),\]
where $\widehat{G}$ is the group of characters of $G=\Gal(K/\QQ).$ Since $G\cong C_2\times C_2$, it follows that any nontrivial character $\chi\in \widehat{G}$ is a quadratic character. Therefore, by the modularity theorem, there exist CM newforms $f_i\in \mathcal{S}_2(\Gamma_0(N_i);\QQ)$, $1\le i\le 4$, for which 
\[\prod_{\chi\in \hat{G}}L(F\otimes \chi,s)=\prod_{i=1}^4 L(f_i,s).\]
Moreover, we have $A\sim \prod_{\chi\in \widehat{G}}F\otimes \chi$, so \[\mathcal{N}_\QQ(A)=\prod_{\chi\in \widehat{G}}\mathcal{N}_\QQ(F\otimes \chi)=\prod_{i=1}^4 N_i.\] 
Now suppose that $L$ is a quadratic field and let $B=\Res_{K/L}(E)$, which is a CM abelian surface over $L$. Then we have that $F$ is CM elliptic curve over $L$ and 
\begin{equation}\label{E:Lprod}
L(E/K,s)=L(B/L,s)=L(F,s)L(F\otimes\chi,s),
\end{equation}
where $\chi$ is a quadratic character on $\Gal(K/L)$. We claim that both $F$ and $F\otimes\chi$ are $L$-isogenous to their Galois conjugates. To see why this claim is true, let $B_1:=\Res_{L/\QQ}(F)$, whose base change to $L$ is isogenous to $F\times F^\sigma$, where $\Gal(L/\QQ)=\langle\sigma\rangle.$ Since $L$ is totally real, we have that $\End_\QQ^0(B_1)=\End_L^0(B_1)\cong \End_L^0(F)\times \End_L^0(F^\sigma)$, which is a $2$-dimensional $\QQ$-algebra. On the other hand, since $\End_\QQ^0(B_1)\cong \End_L^0(F)\times \Hom_L^0(F,F^\sigma)$ and $\End_L^0(F)\cong \QQ$, it follows that $\Hom_L^0(F,F^\sigma)\ne 0$; i.e., $F$ is $L$-isogenous to $F^\sigma$. The same arguments can then be applied to $B_2=\Res_{L/\QQ}(F\otimes \chi)$. Hence we have from (i) that there exist CM newforms $f_i\in \mathcal{S}_2(\Gamma_1(N_i))$ and $g_i\in \mathcal{S}_2(\Gamma_1(M_i)),$ where $i=1,2$, for which $L(F,s)=L(f_1,s)L(f_2,s)$ and $L(F\otimes \chi,s)=L(g_1,s)L(g_2,s)$. The desired result then follows immediately from \eqref{E:Lprod}. Moreover, we have 
\begin{equation*}
A=\Res_{L/\QQ}(B)=B_1\times B_2,
\end{equation*}
so $\mathcal{N}_\QQ(A)=\mathcal{N}_\QQ(B_1)\mathcal{N}_\QQ(B_2)=N_1N_2M_1M_2.$ 

Next, let us assume that $\Res_{K/L}(E)$ is simple for any subfield $L$ of $K$. For a fixed quadratic subfield $L$ of $K$, let $B=\Res_{K/L}(E)$. By the $\ZZ$-module isomorphism 
\[\End_L(B)\cong \prod_{\tau\in \Gal(K/L)}\Hom_K(E^\tau,E),\]
we again have that $\mathcal{D}:=\End_L^0(B)$ is a division algebra of dimension $2$ over $\QQ$, so it must be a field. Note that $\End_\QQ^0A$ admits a $\mathcal{D}$-algebra structure via the following isomorphism \cite[Proposition~4.5]{GQ14}
\begin{equation}\label{E:twist}
\End_\QQ^0(A)=\End_\QQ^0(\Res_{L/\QQ}(B))\cong \mathcal{D}^c[G],
\end{equation}
where the right-hand side denotes the twisted group algebra of $G=\Gal(L/\QQ)$ by a $2$-cocycle $c$ on $G$ with values in $\mathcal{D}^*$. We also have that $\mathcal{D}^c[G]$ has dimension $|G|=2$ over $\mathcal{D}$, so it is commutative. On the other hand, since $A=\Res_{K/\QQ}(E)$ is simple, we can apply the arguments in the proof of (i) to show that $M=\End_\QQ^0(A)$ is a $4$-dimensional division algebra over $\QQ$. This fact together with the commutative property obtained from the isomorphism \eqref{E:twist} implies that $M$ is a quartic number field. Therefore, we can again conclude using Serre's modularity conjecture that $A\sim A_g$ for some CM newform $g\in \mathcal{S}_2(\Gamma_1(N))$, implying
\[L(E,s)=L(A,s)=L(A_g,s)=\prod_{\sigma\in \Gal(M/\QQ)}L(g^\sigma,s),\]
and $\mathcal{N}_\QQ(A)=\mathcal{N}_\QQ(A_g)=N^4.$
\end{proof}

\section{Main results}\label{S:results}

In \cite[Proposition~2.2]{HY22}, He and Ye prove a general formula for Mahler measures of a certain family of three-variable polynomials defining $K3$ surfaces, which is parametrized by a modular function evaluated at CM points, in terms of special values of Dirichlet and modular $L$-functions. Motivated by this result, we start by proving analogous formulas for the families $\mathcal{F}_t$ and $\mathcal{C}_t$.

\begin{proposition}\label{P:MMasLcusp}
	Let $\tau_0=[a,b,c]=-\frac{b}{2a}+i\frac{\sqrt{4ac-b^2}}{2a}\in\mathcal{H}$ be a CM point.
	\begin{itemize}
		\item [(i)] For $\tau_0\in F_P$, let $l=\gcd(a,4b,16c)$ and let $N$ be the smallest positive integer such that the matrix $N\begin{pmatrix}\frac{32c}{l} & \frac{16b}{l}\\ \frac{16b}{l} & \frac{32a}{l}\end{pmatrix}^{-1}$ is \emph{even}; i.e., of the form $\begin{pmatrix}2\ZZ & \ZZ \\ \ZZ & 2\ZZ\end{pmatrix}$. Then
		\[\Theta_{P,\tau_0}(\tau):=\sum_{m,n\in\Z}\chi_{-4}(n)(2bm+an)q^{\frac{16cm^2+4bmn+an^2}{l}}\quad \text{with } q=e^{2\pi i\tau}\]
		is a cusp form in $\mathcal{S}_2\left(\Gamma_0(N),\left(\frac{D}{\cdot}\right)\right)$ with $D=\frac{64(4ac-b^2)}{l^2}$ and we have
		\begin{equation}\label{eq:muasLcusp}
			\mu(t_P(\tau_0))=\frac{8\sqrt{4ac-b^2}}{l^2\pi^2}L(\Theta_{P,\tau_0},2).
		\end{equation}
		\item [(ii)] For $\tau_0\in F_Q$, let $l=\gcd(a,3b,9c)$ and let $N$ be the smallest positive integer such that $N\begin{pmatrix}\frac{18c}{l} & \frac{9b}{l}\\ \frac{9b}{l} & \frac{18a}{l}\end{pmatrix}^{-1}$ is even. Then
		\[\Theta_{Q,\tau_0}(\tau):=\sum_{m,n\in\Z}\chi_{-3}(n)(3bm+2an)q^{\frac{9cm^2+3bmn+an^2}{l}}\quad \text{with } q=e^{2\pi i\tau}\]
		is a cusp form in $\mathcal{S}_2\left(\Gamma_0(N),\left(\frac{D}{\cdot}\right)\right)$ with $D=\frac{27(4ac-b^2)}{l^2}$ and we have
		\begin{equation}\label{eq:nuasLcusp}
			\nu(t_Q(\tau_0))=\begin{cases}
				\frac{27\sqrt{3(4ac-b^2)}}{16l^2\pi^2}L(\Theta_{Q,\tau_0},2), & \text{if }\sqrt[3]{t_Q(\tau_0)}\notin\mathcal{K}^\circ_Q,\\
				\frac{27\sqrt{3(4ac-b^2)}}{4(1-3\sgn(t_Q(\tau_0)))l^2\pi^2}L(\Theta_{Q,\tau_0},2), & \text{if }t_Q(\tau_0)\in(-1,27).
			\end{cases}
		\end{equation}
	\end{itemize}
\end{proposition}
\begin{proof}
	To prove the first part, let $\mathcal{L}$ be the rank $2$ lattice (i.e., the submodule $\Z^2$ of $\R^2$) equipped with the positive definite integral quadratic form
	\[\mathcal{Q}(X)=\frac{1}{2}X^tAX\]
	with $A=\begin{pmatrix}\frac{32c}{l} & \frac{16b}{l}\\ \frac{16b}{l} & \frac{32a}{l}\end{pmatrix}$ and $X=\begin{pmatrix} x_1 \\ x_2 \end{pmatrix}\in\Z^2$. We define the \emph{dual lattice} of $\mathcal{L}$ to be the sub $\Z$-module $\mathcal{L}^*=\{Y\in\R^2: \forall X\in\Z^2, X^{t}AY\in\Z\}$ of $\R^2$. Then by \cite[Corollary 14.3.16]{CS17}, the \emph{generalized theta function}
	\[\Theta(H,\mathcal{L},Y;\tau):=\sum_{X\in\Z^2}H(X+Y)q^{\mathcal{Q}(X+Y)}\]
	is a cusp form in $\mathcal{S}_2(\Gamma(N))$ for every linear homogeneous polynomial $H(x_1,x_2)=a_1x_1+a_2x_2$ and $Y\in\mathcal{L}^*$. It can be seen from the definition that $\Z^2\subset \mathcal{L}^*$ and $\Theta(H,\mathcal{L},Y;\tau)$ only depends on the class of $Y$ in $\mathcal{L}^*/\Z^2$. Since it is obvious that $Y_1=(0,1/4)^t,Y_2=(0,3/4)^t\in \mathcal{L}^*$ and
	\begin{equation*}
		\begin{split}
			\Theta_{P,\tau_0}(\tau) & =  \sum_{m,n\in\Z}(2bm+a(4n+1))q^{\frac{16cm^2+4bm(4n+1)+a(4n+1)^2}{l}} \\
			& \quad -\sum_{m,n\in\Z}(2bm+a(4n+3))q^{\frac{16cm^2+4bm(4n+3)+a(4n+3)^2}{l}} \\
			& = \Theta(H_{\tau_0},\mathcal{L},Y_1;\tau)-\Theta(H_{\tau_0},\mathcal{L},Y_2;\tau),
		\end{split}
	\end{equation*}
	with $H_{\tau_0}(x_1,x_2)=2bx_1+4ax_2$, we immediately see that $\Theta_{P,\tau_0}(\tau)\in\mathcal{S}_2(\Gamma(N))$. To prove that $\Theta_{P,\tau_0}(\tau)\in\mathcal{S}_2(\Gamma_0(N),\left(\frac{D}{\cdot}\right))$, one needs to show that
	\begin{equation}\label{eq:goal}
		(c_0\tau+d_0)^{-2}\Theta_{P,\tau_0}(\gamma\tau)=\left(\frac{D}{d_0}\right)\Theta_{P,\tau_0}(\tau)\quad \text{for all }\gamma=\begin{pmatrix} a_0 & b_0 \\ c_0 & d_0 \end{pmatrix}\in \Gamma_0(N).
	\end{equation}

	Since for any $\gamma=\begin{pmatrix} a_0 & b_0 \\ c_0 & d_0 \end{pmatrix}\in \Gamma_0(N)$, $\Theta(H,\mathcal{L},Y;\tau)$ satisfies the following transformation formula (see \cite[Chapter IX, \S 4, Theorem 5]{Schoeneberg74}):
	\[(c_0\tau+d_0)^{-2}\Theta(H,\mathcal{L},Y;\gamma\tau)=\left(\frac{-\det{A}}{d_0}\right)e^{2\pi ia_0b_0\mathcal{Q}(Y)}\Theta(H,\mathcal{L},a_0Y;\tau),\]
	we have
	\begin{equation*}
		\begin{split}
			(c_0\tau+d_0)^{-2}\Theta_{P,\tau_0}(\gamma\tau) & = (c_0\tau+d_0)^{-2}\bigl(\Theta(H_{\tau_0},\mathcal{L},Y_1;\gamma\tau)-\Theta(H_{\tau_0},\mathcal{L},Y_2;\gamma\tau)\bigr) \\
			& = \left(\frac{-\det{A}}{d_0}\right)(\Theta(H_{\tau_0},\mathcal{L},a_0Y_1;\tau)-\Theta(H_{\tau_0},\mathcal{L},a_0Y_2;\tau)).
		\end{split}
	\end{equation*}
	Note that $N$ must be divisible by $4$ because $\det(A^{-1})=\frac{l^2}{256(4ac-b^2)}$ and $NA^{-1}$ has integer determinant. Thus, we have $\gcd(a_0,4)=1$ and
	\begin{equation*}
		\begin{split}
			(c_0\tau+d_0)^{-2}\Theta_{P,\tau_0}(\gamma\tau) & = \left(\frac{-\det{A}}{d_0}\right)\left(\frac{-4}{a_0}\right)(\Theta(H_{\tau_0},\mathcal{L},Y_1;\tau)-\Theta(H_{\tau_0},\mathcal{L},Y_2;\tau)) \\
			& = \left(\frac{\det A/4}{d_0}\right)\left(\frac{-4}{a_0d_0}\right)\Theta_{P,\tau_0}(\tau)\\
			& = \left(\frac{D}{d_0}\right)\Theta_{P,\tau_0}(\tau),
		\end{split}
	\end{equation*}
	where the first and the last equalities follow by the fact that
	\[(a_0Y_1,a_0Y_2)\equiv\begin{cases}
		(Y_1,Y_2)\mod\Z^2, & \text{if } a_0\equiv 1\mod 4,\\
		(Y_2,Y_1)\mod\Z^2, & \text{if }a_0\equiv 3\mod 4
	\end{cases}\]
	and $a_0d_0\equiv1\ (\mathrm{mod}\ 4)$. This completes our verification of \eqref{eq:goal}. Now, it is straightforward to obtain \eqref{eq:muasLcusp} by substituting $\tau_0=-\frac{b}{2a}+i\frac{\sqrt{4ac-b^2}}{2a}$ into \eqref{eq:RVformu}.

	As for the second part, we can turn to the lattice $\mathcal{L}$ associated to the quadratic form $\mathcal{Q}(X)=\frac{1}{2}X^tAX$ with $A=\begin{pmatrix}\frac{18c}{l} & \frac{9b}{l}\\ \frac{9b}{l} & \frac{18a}{l}\end{pmatrix}$ and express $\Theta_{Q,\tau_0}(\tau)$ as
	\[\Theta(H_{\tau_0},\mathcal{L},Y_1;\tau)-\Theta(H_{\tau_0},\mathcal{L},Y_2;\tau)\]
	with $H_{\tau_0}(x_1,x_2)=3bx_1+6ax_2$ and $Y_1=(0,1/3)^t,Y_2=(0,2/3)^t\in\mathcal{L}^*$. We can then proceed with the rest of the proof in the same way as above, provided that we use \eqref{eq:RVfornu} to obtain \eqref{eq:nuasLcusp} in the final step.
\end{proof}

For all $t\in\mathsf{T}_{P,4}$ (resp. $\mathsf{T}_{Q,4}$), we also include the corresponding identities \eqref{eq:muasLcusp} (resp. \eqref{eq:nuasLcusp}) in the data file \cite{ST24}.

\begin{theorem}\label{T:2x2det}
	For each of the $12$ (resp. $15$) pairs of real quadratic $\{t,t^\sigma\}\subset\mathsf{T}_{P,4}$ (resp. $\mathsf{T}_{Q,4}$) with $\Gal(\Q(t)/\Q)=\langle \sigma \rangle$, there exist $a,b\in \Z\setminus\{0\},a>0$ and $r,s\in \Q$ such that

	\begin{equation}\label{2x2Gamma04}
		\det\begin{pmatrix}
		a\mu(t) & b\mu(t^\sigma)\\
		-\mu(t^\sigma) & \mu(t)
		\end{pmatrix}=\frac{r}{\pi^4}L(\mathcal{F}_t,2)=sL''(\mathcal{F}_t,0).
	\end{equation}

	\begin{equation}\label{2x2Gamma03}
		\left(\text{resp. }\det\begin{pmatrix}
			a\nu(t) & b\nu(t^\sigma)\\
			-\nu(t^\sigma) & \nu(t)
			\end{pmatrix}=\frac{r}{\pi^4}L(\mathcal{C}_t,2)=sL''(\mathcal{C}_t,0).\right)
	\end{equation}
	These identities are labeled $\#1$ to $\#12$ (resp. $\#13$ to $\#27$) in Table \ref{2x2identities_table} and we use bold text to indicate values in $(-1,27)$.
\end{theorem}

\begin{proof}[Proof (Sketch)]
	For $t=8+6\sqrt{2}$ and $t^\sigma=8-6\sqrt{2}$ (see $\#1$ in Table \ref{2x2identities_table}), the first author proved in \cite[Theorem~2.2]{Samart15} that
	\begin{equation}\label{eq:no.1_MM_as_Lcusp}
		\mu(t)=\frac{4}{\pi^2}(L_{\texttt{32aa1}}(2)+2L_{\texttt{64aa1}}(2)),\quad \mu(t^\sigma)=\frac{4}{\pi^2}(-L_{\texttt{32aa1}}(2)+2L_{\texttt{64aa1}}(2)),
	\end{equation}
where $L_{\texttt{32aa1}}(s)$ and $L_{\texttt{64aa1}}(s)$ are $L$-functions of the newforms $f_\texttt{32aa1}$ and $f_\texttt{64aa1}$ respectively (see Table~\ref{newforms_table} for their LMFDB lebels). Also note that
	\begin{equation}\label{eq:proven2x2}
		\#4, \#6, \#7, \#9, \#10, \#12, \text{ and } \#16
	\end{equation}
	have been proven in \cite{GJLQ24}, \cite{TG25}, and \cite{TGW24}. The proofs of \eqref{eq:proven2x2} rely on the same method, which is summarized below.

	Let $E$ be an elliptic curve over some real quadratic field $K=\Q(t)$ in the family $\mathcal{F}_t$ or $\mathcal{C}_t$ with $\Gal(K/\Q)=\langle\sigma\rangle$ and assume that it is $K$-isogenous to $E^\sigma$ via $\phi:E\to E^\sigma$. Then we choose some nontrivial Deninger paths $\gamma_1\in H_1(E(\CC),\ZZ)^-,\gamma_2\in H_1(E^\sigma(\CC),\ZZ)^-$ and let $M_1=\{f,g\},M_2=\phi^*M_1^\sigma\in K_2^T(E)$ with $f=x^2,g=y^2$ for $\mathcal{F}_t$ and $f=x^3,g=y$ for $\mathcal{C}_t$ (as indicated in Section \ref{S:lemmas}). Under this setting, the regulator $|\det(\langle \gamma_i,M_j \rangle)_{1\le i,j\le 2}|$ should be of the form
	\begin{equation}\label{eq:2x2regulator}
		\left|\det\begin{pmatrix}
			\Q^*\m_1 & \Q^*\m_2\\
			\Q^*\m_2 & \Q^*\m_1
		\end{pmatrix}\right|=a\m_1^2+b\m_2^2
	\end{equation}
	for some $a,b\in\Q^*$ with $\m_1=\mu(t), \m_2=\mu(t^\sigma)$ for $\mathcal{F}_t$ and $\m_1=\nu(t),\m_2=\nu(t^\sigma)$ for $\mathcal{C}_t$. Beilinson's conjecture then predicts that \eqref{eq:2x2regulator} should be some non-zero rational multiple of $\frac{1}{\pi^4}L(E,2)$, which can be written in terms of  $L''(E,0)$ via \eqref{E:FE}. Then the identities in \eqref{eq:proven2x2} are proven by explicitly constructing the entries in the $2\times 2$ matrix of \eqref{eq:2x2regulator}. This process usually requires a lot of complicated calculations.
	
	To handle the remaining identities in this theorem, as mentioned in Section~\ref{S:Intro}, we will take a more computational approach. Let $E$ be a curve in the families $\mathcal{F}_t$ or $\mathcal{C}_t$ with the corresponding $t$ in Table \ref{2x2identities_table}. Then $E$ has CM and one can verify case-by-case using \textsc{SageMath} that $E$ is isogenous over $K=\Q(t)$ to $E^\sigma$, where $\Gal(K/\Q)=\langle\sigma\rangle$. For each such pair of $t$ and $t^\sigma$, we can calculate $\m_1,\m_2, L(E,2)$ numerically and then use \textsf{PSLQ} (via the command \texttt{FindIntegerNullVector} in \textsc{Mathematica} with working precision set to $50$ decimal places) to detect the integer relation among $\m_1^2,\m_2^2,$ and $\frac{1}{\pi^4}L(E,2)$. It turns out that in each case, we obtain an identity of the form
	\begin{equation}\label{eq:2x2detected}
		a\m_1^2+b\m_2^2\overset{?}{=}\frac{r}{\pi^4}L(E,2),\quad a,b,r\in\Z\setminus\{0\}.
	\end{equation}
	Once we know a priori the expected identity \eqref{eq:2x2detected}, proving it rigorously will be a matter of routine. By Proposition \ref{P:Lsplit}, there are newforms $f_i\in\mathcal{S}_2(\Gamma_1(N_i)), i=1,2$ such that
	\[L(E,s)=L(f_1,s)L(f_2,s).\]
	It turns out that these newforms can be identified easily with the aid of LMFDB. We collect the $L$-function identities in Table~\ref{L_as_products_table}, where $L_\texttt{Nxxx}$ in the table denotes the $L$-function of the level $N$ newform $f_\texttt{Nxxx}$ collected in Table~\ref{newforms_table}. Let $\tau_1$ and $\tau_2$ be the CM points that correspond to $t$ and $t^\sigma$ respectively in Table \ref{2x2identities_table}. In each case, the theta series $\Theta_{P,\tau_i}$ or $\Theta_{Q,\tau_i}$ in Proposition \ref{P:MMasLcusp} always have the same level $N$ with $N_i\mid N, i=1,2$ and we can use the Sturm bound \cite[Theorem~1]{Sturm87} to prove that these theta series are linear combinations of the cusp forms in
	\[\bigcup_{i=1,2}\{f_i(d\tau) : d\mid(N/N_i)\}.\]
	Then by Proposition \ref{P:MMasLcusp}, we can express $\m_i$ as linear combinations of $L$-values of $f_i$ and \eqref{eq:2x2detected} follows immediately due to cancellations of the remaining terms. Instead of working out the computational details for all identities, we will provide some examples in Section~\ref{S:examples} to illustrate the whole process in different situations (see Example~\ref{ex:2x2_no.1}--Example~\ref{ex:2x2_no.23}). 
\end{proof}
\begin{remark}\label{r:PSLQ}
It is worth noting that one can also anticipate an equivalent form of \eqref{eq:2x2detected} using the information above. Namely, suppose $\m_1$ and $\m_2$ are explicitly written in terms of $L_1:=L'(f_1,0)$ and $L_2:=L'(f_2,0)$. Then $a\m_1^2+b\m_2^2$ is a linear combination of $L_1,L_1L_2,$ and $L_2^2$. Imposing the expected condition that the coefficients of $L_1^2$ and $L_2^2$ vanish immediately determines $a/b$. While the \textsf{PSLQ} approach is not truly needed in this case, it is very helpful and efficient for computations in the biquadratic case, which involves many more unknown coefficients.
\end{remark}

To state the next theorem, let us introduce the following notation for conciseness of the statements. For $a,b,c,d\in \N$ and $t_1,t_2,t_3,t_4\in \CC$, define 
\begin{equation}\label{eq:matrix_for_collect}
\mathfrak{M}_{a,b,c,d}:=\mathfrak{M}_{a,b,c,d}(t_1,t_2,t_3,t_4)=\det\begin{pmatrix}
\sqrt{a} \mu(t_1) & \sqrt{b} \mu(t_2) & \sqrt{c} \mu(t_3) & \sqrt{d} \mu(t_4)\\
\sqrt{b} \mu(t_2) & \sqrt{a} \mu(t_1) & \sqrt{d} \mu(t_4) & \sqrt{c} \mu(t_3)\\
\sqrt{c} \mu(t_3) & \sqrt{d} \mu(t_4) & \sqrt{a} \mu(t_1) & \sqrt{b} \mu(t_2)\\
\sqrt{d} \mu(t_4) & \sqrt{c} \mu(t_3) & \sqrt{b} \mu(t_2) & \sqrt{a} \mu(t_1)
\end{pmatrix}.
\end{equation}
In addition, for each $1\le j\le 6$, we define $\mathfrak{M}^{[j]}_{a,b,c,d}:=\mathfrak{M}^{[j]}_{a,b,c,d}(t_1,t_2,t_3,t_4)$ as the determinant $\mathfrak{M}_{a,b,c,d}$ with the following sign changes:
\begin{equation}\label{eq:sign_changes}
	\begin{split}
		\mathfrak{M}^{[1]}_{a,b,c,d}&= \begin{vmatrix}
			+ & + & + & +\\
			+ & + & + & +\\
			+ & + & - & -\\
			+ & + & - & -
			\end{vmatrix}, \ \mathfrak{M}^{[2]}_{a,b,c,d}= \begin{vmatrix}
			+ & + & + & +\\
			+ & - & - & +\\
			+ & - & + & -\\
			+ & + & - & -
			\end{vmatrix}, \ \mathfrak{M}^{[3]}_{a,b,c,d}= \begin{vmatrix}
			+ & + & + & +\\
			+ & - & + & -\\
			+ & + & + & +\\
			+ & - & + & -
			\end{vmatrix},\\
		\mathfrak{M}^{[4]}_{a,b,c,d}&= \begin{vmatrix}
			+ & + & + & +\\
			+ & - & + & -\\
			+ & + & - & -\\
			+ & - & - & +
			\end{vmatrix}, \ \mathfrak{M}^{[5]}_{a,b,c,d}= \begin{vmatrix}
			+ & + & + & +\\
			+ & - & - & +\\
			+ & - & - & +\\
			+ & + & + & +
			\end{vmatrix}, \ \mathfrak{M}^{[6]}_{a,b,c,d}= \begin{vmatrix}
			+ & + & + & +\\
			+ & + & - & -\\
			+ & - & + & -\\
			+ & - & - & +
			\end{vmatrix}.
	\end{split}
\end{equation}
Then we define $\mathfrak{N}^{[j]}_{a,b,c,d}:=\mathfrak{N}^{[j]}_{a,b,c,d}(t_1,t_2,t_3,t_4)$ in the same manner as $\mathfrak{M}^{[j]}_{a,b,c,d}$, but replace $\mu(t_i)$ with $\nu(t_i)$. (Note that there can be many different ways to assign the sign for each entry in the matrix without changing its determinant; these assignments are chosen as our preferred choice.)

\begin{theorem}\label{T:4x4det}
	For each of the $24$ (resp. $24$) totally real quartic Galois orbits $\{t_1,t_2,t_3,t_4\}\subset\mathsf{T}_{P,4}$ (resp. $\mathsf{T}_{Q,4}$), there exist $a,b,c,d\in\N$ and $r,s\in \Q$ such that, for each $t\in \{t_1,t_2,t_3,t_4\}$,
	\begin{equation}\label{det4x4Gamma04}
		\mathfrak{M}^{[j]}_{a,b,c,d}(t_1,t_2,t_3,t_4)=\frac{r}{\pi^8}L(\mathcal{F}_{t},2)=sL^{(4)}(\mathcal{F}_{t},0)
	\end{equation}

	\begin{equation}\label{det4x4Gamma03}
		\left(\text{resp. }\mathfrak{N}^{[j]}_{a,b,c,d}(t_1,t_2,t_3,t_4)=\frac{r}{\pi^8}L(\mathcal{C}_{t},2)=sL^{(4)}(\mathcal{C}_{t},0)\right)
	\end{equation}
	for some $1\le j\le 6$. These Galois orbits and identities are labeled $\#28$ to $\#51$ (resp. $\#52$ to $\#75$) in Table~\ref{quarticGalorbs_table} and Table~\ref{4x4identities_table}, respectively, and we use bold text in Table \ref{quarticGalorbs_table} to indicate values in $(-1,27)$.
\end{theorem}

\begin{proof}[Proof (Sketch)]
	Let $\{t_1,t_2,t_3,t_4\}$ be one of the Galois orbits in Table \ref{quarticGalorbs_table} and let $E$ be the CM curve $\mathcal{F}_{t_1}$ or $\mathcal{C}_{t_1}$ over $K=\Q(t_1)$. Hence, for $i=2,3,4$, the curve $\mathcal{F}_{t_i}$ or $\mathcal{C}_{t_i}$
	is given by $E^{\sigma_i}$ for some nontrivial $\sigma_i\in \Gal(K/\QQ)$. Also, let $\sigma_1\in\Gal(K/\QQ)$ be the trivial element. One can verify using \textsc{SageMath} that $E^{\sigma_i}$'s are in the same isogeny class; i.e., there exist $K$-isogenies $\phi_i:E\to E^{\sigma_i}$ for $i=1,2,3,4$.

	If we choose $M_1=\{f,g\}\in K_2^T(E)$, as given in the proof of Theorem~\ref{T:2x2det}, and some nontrivial $\gamma_i\in H_1(E^{\sigma_i}(\CC),\ZZ)^-$ for all $i=1,2,3,4$, then we can construct $M_i=\phi_i^*\{f^{\sigma_i},g^{\sigma_i}\}\in K_2^T(E),i=2,3,4$ and the regulator \eqref{eq:nxn_regulator} for $\gamma_1,\dots\gamma_4$ and $M_1,\dots, M_4$ reads
	\begin{equation}\label{eq:4x4regulator}
		\left|\det\left(\frac{1}{2\pi}\int_{\gamma_i}\eta\left((\phi_j^*f^{\sigma_j})^{\sigma_i},(\phi_j^*g^{\sigma_j})^{\sigma_i}\right)\right)_{1\le i,j\le 4}\right|.
	\end{equation}
	Note that
	\begin{equation*}
		\begin{split}
			\int_{\gamma_i}\eta\left((\phi_j^*f^{\sigma_j})^{\sigma_i},(\phi_j^*g^{\sigma_j})^{\sigma_i}\right) & = \int_{\gamma_i}\eta\left((\phi_j^{\sigma_i})^*f^{\sigma_j\sigma_i},(\phi_j^{\sigma_i})^*g^{\sigma_j\sigma_i}\right) \\
			& = \int_{\phi_j^{\sigma_i}\circ\gamma_i}\eta(f^{\sigma_j\sigma_i},g^{\sigma_j\sigma_i}),
		\end{split}
	\end{equation*}
	where $\phi_j^{\sigma_i}\circ\gamma_i$ is the path on $E^{\sigma_j\sigma_i}(\CC)$ obtained by pushing forward the path $\gamma_i$ on $E^{\sigma_i}(\CC)$ via the isogeny $\phi_j^{\sigma_i}:E^{\sigma_i}\to E^{\sigma_j\sigma_i}$. Since we have $\Gal(K/\Q)\cong C_2\times C_2$, one can calculate that \eqref{eq:4x4regulator} should be of the form 
	\begin{equation}\label{eq:4x4_Q_matrix}
		\left|\det\begin{pmatrix}
			\Q^*\m_1 & \Q^*\m_2 & \Q^*\m_3 & \Q^*\m_4 \\
			\Q^*\m_2 & \Q^*\m_1 & \Q^*\m_4 & \Q^*\m_3 \\
			\Q^*\m_3 & \Q^*\m_4 & \Q^*\m_1 & \Q^*\m_2 \\
			\Q^*\m_4 & \Q^*\m_3 & \Q^*\m_2 & \Q^*\m_1 
		\end{pmatrix}\right|,
	\end{equation}
	where $\m_i=\mu(t_i)$ for $\mathcal{F}_t$ and $\nu(t_i)$ for $\mathcal{C}_t$. By expanding the determinant and assuming Beilinson's conjecture, we expect that there exists an identity of the form
	\begin{equation}\label{eq:4x4regulator_expanded}
		\begin{split}
			& a_1\m_1^4+a_2\m_2^4+a_3\m_3^4+a_4\m_4^4+a_{12}\m_1^2\m_2^2+a_{13}\m_1^2\m_3^2+a_{14}\m_1^2\m_4^2  \\
			& \quad +a_{23}\m_2^2\m_3^2+a_{24}\m_2^2\m_4^2+a_{34}\m_3^2\m_4^2+a_{1234}\m_1\m_2\m_3\m_4 =\frac{r}{\pi^8}L(E,2)
		\end{split}
	\end{equation}
	with integer coefficients. By (the second part of) Proposition \ref{P:Lsplit}, there are CM newforms $f_i\in\mathcal{S}_2(\Gamma_1(N_i)), i=1,2,3,4$ such that
	\begin{equation}\label{eq:Lquartic}
		L(E,s)=L(f_1,s)L(f_2,s)L(f_3,s)L(f_4,s).
	\end{equation}
	Given the conductor of $\Res_{K/\QQ}(E)$, which can be computed using Milne's formula \cite[Proposition~1]{Milne72}, and the fact that the above newforms are CM, we can explicitly identify these newforms with the aid of LMFDB and \textsc{PARI/GP} \cite{PARI2}. The $L$-function identities are also collected in Table~\ref{L_as_products_table}. Unlike the quadratic case, to apply \textsf{PSLQ}, it seems impossible to directly compute $L(E,2)$ with sufficient precision in timely fashion, especially when the conductor norm of $E$ is large. As a byproduct of \eqref{eq:Lquartic}, we can calculate $L(E,2)$ with high precision using numerical values of $L(f_i,2)$, which can be computed rapidly on \textsc{PARI/GP}. Then we can use \textsf{PSLQ} to detect an identity of the form \eqref{eq:4x4regulator_expanded}. In our practice, we can always obtain such an (correct) identity in \textsc{Mathematica} with $250$ decimal places. To validate the identity, we use the same procedure as in the proof of Theorem~\ref{T:2x2det}. Let $\tau_i$ be the CM point corresponds to $t_i$ in Table \ref{quarticGalorbs_table}. One can check that the theta series $\Theta_{P,\tau_i}$ or $\Theta_{Q,\tau_i}$ all have the same level $N$ with $N_i\mid N$ and can be expressed as linear combinations of the cusp forms in
	\[\bigcup_{i=1,2,3,4}\{f_i(d\tau) : d\mid(N/N_i)\}.\]
	Using Proposition~\ref{P:MMasLcusp}, we can then express $\m_i$ as linear combinations of $L(f_i,2)$ and verify \eqref{eq:4x4regulator_expanded} after proper cancellations. It turns out that the left-hand side of \eqref{eq:4x4regulator_expanded} can always be written as a determinant of the form \eqref{eq:matrix_for_collect}, up to six types of sign changes presented in \eqref{eq:sign_changes}. Similar to the quadratic case, we will present some illustrative examples in Section~\ref{S:examples} (see Example~\ref{ex:4x4_no.28}--Example~\ref{ex:4x4_no.71}).
\end{proof}

\begin{remark}\label{rk:4x4shape_and_missingforms}
\begin{itemize}
\item[(i)] The reader may notice that entries of the matrix in \eqref{eq:matrix_for_collect} may have irrational coefficients and thus look different from those in \eqref{eq:4x4_Q_matrix}. To better align with Beilinson's conjecture, in each case, we can in fact easily recover a $4\times 4$ determinant with \emph{integer} coefficients by multiplying or dividing certain rows and columns of the determinant collected in Table~\ref{4x4identities_table} by values in $\{\sqrt{a},\sqrt{b},\sqrt{c},\sqrt{d}\}$. For example, the matrices in \eqref{E:mmquartic1} and \eqref{E:mmquartic2} can be recovered from
	\[\begin{pmatrix}
		\sqrt{3}\mu(t_1) & \mu(t_2) & \mu(t_3) & \sqrt{3}\mu(t_4)\\
		\mu(t_2) & -\sqrt{3}\mu(t_1) & -\sqrt{3}\mu(t_4) & \mu(t_3)\\
		\mu(t_3) & -\sqrt{3}\mu(t_4) & -\sqrt{3}\mu(t_1) & \mu(t_2)\\
		\sqrt{3}\mu(t_4) & \mu(t_3) & \mu(t_2) & \sqrt{3}\mu(t_1)
	\end{pmatrix}\]
	and
	\[\begin{pmatrix}
		\sqrt{185}\nu(t_1) & \sqrt{37}\nu(t_2) & \sqrt{5}\nu(t_3) & \nu(t_4)\\
		\sqrt{37}\nu(t_2) & -\sqrt{185}\nu(t_1) & \nu(t_4) & -\sqrt{5}\nu(t_3)\\
		\sqrt{5}\nu(t_3) & \nu(t_4) & \sqrt{185}\nu(t_1) & \sqrt{37}\nu(t_2)\\
		\nu(t_4) & -\sqrt{5}\nu(t_3) & \sqrt{37}\nu(t_2) & -\sqrt{185}\nu(t_1)
	\end{pmatrix}\]
	respectively. For the sake of compactness, we choose to present our results for $4\times 4$ determinants in the form \eqref{eq:matrix_for_collect} instead of those with integer coefficients.
\item [(ii)] Since the newforms in the character orbit spaces \href{https://www.lmfdb.org/ModularForm/GL2/Q/holomorphic/2124/2/g/}{\texttt{2124.2.g}}, \href{https://www.lmfdb.org/ModularForm/GL2/Q/holomorphic/2385/2/g/}{\texttt{2385.2.g}}, and \href{https://www.lmfdb.org/ModularForm/GL2/Q/holomorphic/3712/2/g/}{\texttt{3712.2.g}} have not yet been added to the LMFDB (as of September 8, 2024), the newforms that correspond to the $L$-functions in bold text in Table~\ref{L_as_products_table} ($\#50, \#72,$ and $\#73$) are not collected in Table \ref{newforms_table}. They are the unique newforms in \href{https://www.lmfdb.org/ModularForm/GL2/Q/holomorphic/3712/2/g/}{\texttt{3712.2.g}}, \href{https://www.lmfdb.org/ModularForm/GL2/Q/holomorphic/2385/2/g/}{\texttt{2385.2.g}}, and \href{https://www.lmfdb.org/ModularForm/GL2/Q/holomorphic/2124/2/g/}{\texttt{2124.2.g}} with coefficient fields $\Q(\sqrt{-2},\sqrt{29}), \Q(\sqrt{-5},\sqrt{-53})$, and $\Q(\sqrt{-2},\sqrt{-59})$, respectively. One can use \textsc{PARI/GP} to compute their Fourier coefficients. See Example~\ref{ex:4x4_no.50} for the details for $\#50$.
\end{itemize}
\end{remark} 

By using the data collected in \cite{ST24}, and with the help of Sturm bounds, we can also prove identities expressing $\mu(t)$ or $\nu(t)$ as linear combinations of $L$-values of newforms for non-totally real $t$. For example, by limiting the degree of $t$ to no more than $3$, we can prove:

\begin{theorem}\label{thm:MM_for_N3}
	For each non-totally real $t\in\mathsf{T}_{P,3}$ or $\mathsf{T}_{Q,3}$, we can express $\mu(t)$ or $\nu(t)$ in terms of $L$-values of newforms as follows:

	\begin{equation}\label{eq:mu88}
		\mu(8\pm8i\sqrt{3})=\frac{1}{6}\big((3+i\sqrt{3})L'_\textnormal{\texttt{48ca1}}(0)+(3-i\sqrt{3})L'_{\textnormal{\texttt{48ca2}}}(0)\big),
	\end{equation}

	\begin{equation}
		\mu(8\pm24i\sqrt{7})=\frac{1}{14}\big((21+5i\sqrt{7})L'_{\textnormal{\texttt{28da1}}}(0)+(21-5i\sqrt{7})L'_{\textnormal{\texttt{28da2}}}(0)\big),
	\end{equation}

	\begin{equation}
		\mu\biggl(\frac{1}{2}\pm\frac{3i\sqrt{7}}{2}\biggr)=\frac{1}{14}\big((7+3i\sqrt{7})L'_{\textnormal{\texttt{28da1}}}(0)+(7-3i\sqrt{7})L'_{\textnormal{\texttt{28da2}}}(0)\big),
	\end{equation}

	\begin{equation}\label{eq:2x2_N2_28da}
		\mu\biggl(\frac{31}{2}\pm\frac{3i\sqrt{7}}{2}\biggr)=\frac{1}{7}\big((7+i\sqrt{7})L'_{\textnormal{\texttt{28da1}}}(0)+(7-i\sqrt{7})L'_{\textnormal{\texttt{28da2}}}(0)\big),
	\end{equation}

	\begin{equation}
		\nu(4\pm10i\sqrt{2})=\frac{3}{8}\big((2+i\sqrt{2})L'_{\textnormal{\texttt{24fa1}}}(0)+(2-i\sqrt{2})L'_{\textnormal{\texttt{24fa2}}}(0)\big),
	\end{equation}

	\begin{equation}
		\nu(32\pm 8i\sqrt{11})=\frac{3}{44}\big((11+i\sqrt{11})L'_{\textnormal{\texttt{33da1}}}(0)+(11-i\sqrt{11})L'_{\textnormal{\texttt{33da2}}}(0)\big),
	\end{equation}

	\begin{equation}\label{eq:3x3Samart15}
		\nu(6-6\sqrt[3]{2}+18\sqrt[3]{4})=\frac{1}{2}\big(-3L'_{\textnormal{\texttt{27aa1}}}(0)+L'_{\textnormal{\texttt{36aa1}}}(0)+L'_{\textnormal{\texttt{108aa1}}}(0)\big),
	\end{equation}

	\begin{equation}\label{eq:3x3_N3_27_108}
		\nu(6-6\rho\sqrt[3]{2}+18\bar{\rho}\sqrt[3]{4})=\frac{1}{4}\big(3L'_{\textnormal{\texttt{27aa1}}}(0)+L'_{\textnormal{\texttt{108aa1}}}(0)\big),
	\end{equation}

	\begin{equation}\label{eq:3x3Samart15TG241}
		\nu(17766+14094\sqrt[3]{2}+11178\sqrt[3]{4})=\frac{1}{2}\big(3L'_{\textnormal{\texttt{27aa1}}}(0)+3L'_{\textnormal{\texttt{36aa1}}}(0)+L'_{\textnormal{\texttt{108aa1}}}(0)\big),
	\end{equation}

	\begin{equation}\label{eq:3x3Samart15TG242}
		\nu(17766+14094\rho\sqrt[3]{2}+11178\bar{\rho}\sqrt[3]{4})=\frac{1}{2}\big(-6L'_{\textnormal{\texttt{27aa1}}}(0)+3L'_{\textnormal{\texttt{36aa1}}}(0)+L'_{\textnormal{\texttt{108aa1}}}(0)\big),
	\end{equation}

	\begin{equation}
		\nu(96+56\sqrt[3]{3}-72\sqrt[3]{9}) = -L'_{\textnormal{\texttt{27aa1}}}(0)-L'_{\textnormal{\texttt{243aa1}}}(0)-L'_{\textnormal{\texttt{243ab1}}}(0),
	\end{equation}

	\begin{equation}\label{eq:3x3_N3_243}
		\nu(96+56\rho\sqrt[3]{3}-72\bar{\rho}\sqrt[3]{9})=\frac{1}{6}\big(-L'_{\textnormal{\texttt{243aa1}}}(0)+L'_{\textnormal{\texttt{243ab1}}}(0)\big),
	\end{equation}

	\begin{equation}
		\nu\big(216(-18964-13149\sqrt[3]{3}-9117\sqrt[3]{9})\big)=\frac{1}{3}\big(9L'_{\textnormal{\texttt{27aa1}}}(0)-L'_{\textnormal{\texttt{243aa1}}}(0)+L'_{\textnormal{\texttt{243ab1}}}(0)\big),
	\end{equation}

	\begin{equation}\label{eq:3x3_N3_27}
		\nu\big(216(-18964-13149\rho\sqrt[3]{3}-9117\bar{\rho}\sqrt[3]{9})\big) = \frac{1}{6}\big(-9L'_{\textnormal{\texttt{27aa1}}}(0)+4L'_{\textnormal{\texttt{243aa1}}}(0)+5L'_{\textnormal{\texttt{243ab1}}}(0)\big),
	\end{equation}
	where $\rho=e^{\pm\frac{2\pi i}{3}}$.
\end{theorem}

\begin{remark}
	Note that \eqref{eq:3x3Samart15} is proven in \cite[Theorem~2.9]{Samart15}; \eqref{eq:3x3Samart15TG241} and \eqref{eq:3x3Samart15TG242} are conjectured in \cite{Samart15} and proven in \cite[Corollary~1.3]{TG25}; the Mahler measures in \eqref{eq:mu88}--\eqref{eq:2x2_N2_28da}, \eqref{eq:3x3_N3_27_108}, and \eqref{eq:3x3_N3_243}--\eqref{eq:3x3_N3_27} are expressed as $L$-values of theta functions in \cite{TG25,TGW24}.	
\end{remark}

We conclude this section by providing some results of this kind with non-totally real quartic $t$. Similar identities for the remaining values of $t$ in \cite{ST24} can be deduced using the method mentioned above.

\begin{theorem}
	The following identities are true:
	\begin{equation*}
		\mu(8\pm6\sqrt[4]{2}\mp9\sqrt[4]{8}) = \frac{1}{4}\big(L'_{\textnormal{\texttt{32aa1}}}(0)+L'_{\textnormal{\texttt{64aa1}}}(0)\mp L'_{\textnormal{\texttt{256ab1}}}(0)\mp L'_{\textnormal{\texttt{256ac1}}}(0)\big),
	\end{equation*}

	\begin{equation*}
		\mu(8\pm6i\sqrt[4]{2}\pm9i\sqrt[4]{8}) = \frac{1}{8}\big(-L'_{\textnormal{\texttt{256ab1}}}(0)+ L'_{\textnormal{\texttt{256ac1}}}(0)\big),
	\end{equation*}

	\begin{equation*}
		\begin{split}
			\nu\big(-73-70\sqrt{2}\pm i(161+115\sqrt{2})\big) & =\frac{3}{32}\big((8+i\sqrt{2})L'_{\textnormal{\texttt{24fa1}}}(0)+(8-i\sqrt{2})L'_{\textnormal{\texttt{24fa2}}}(0) \\
			& \quad +(2+i\sqrt{2})L'_{\textnormal{\texttt{96fa1}}}(0)+(2-i\sqrt{2})L'_{\textnormal{\texttt{96fa2}}}(0)\big),
		\end{split}
	\end{equation*}

	\begin{equation*}
		\begin{split}
			\nu\big(-73+70\sqrt{2}\pm i(161-115\sqrt{2})\big) & = \frac{3}{32}\big((-10+7i\sqrt{2})L'_{\textnormal{\texttt{24fa1}}}(0)+(-10-7i\sqrt{2})L'_{\textnormal{\texttt{24fa2}}}(0) \\
			& \quad +(4-i\sqrt{2})L'_{\textnormal{\texttt{96fa1}}}(0)+(4+i\sqrt{2})L'_{\textnormal{\texttt{96fa2}}}(0)\big).
		\end{split}
	\end{equation*}
\end{theorem}

\section{Examples}\label{S:examples}
In this section, we give some examples illustrating results in Table~\ref{2x2identities_table} and Table~\ref{4x4identities_table}. In each example, the Sturm bound is crucial for identifying the cusp forms $\Theta_{P,\tau_i}(\tau)$ or $\Theta_{Q,\tau_i}(\tau)$ as linear combinations of those in Table~\ref{newforms_table}, so we briefly recall its definition here. The Sturm bound for the space $\mathcal{M}_k\left(\Gamma_0(N)\right)$ is the integer
\[B(k,N):=\left\lfloor\frac{k m}{12}\right\rfloor,\]
where $\displaystyle m:=\left[\mathrm{SL}_2(\mathbb{Z}): \Gamma_0(N)\right]=N \prod_{p \mid N}\left(1+\frac{1}{p}\right).$ For $k>1$, if $f$ and $g$ are elements of $\mathcal{M}_k\left(\Gamma_0(N)\right)$ and their first $B(k,N)$ Fourier coefficients coincide, then $f=g$ (see \cite[Corollary~9.20]{Stein07}).
\begin{example}[$\#1$]\label{ex:2x2_no.1}
	Let $\tau_1=[16,0,1],\tau_2=[16,16,5]$. Then we have
	\[t=t_P(\tau_1)=8+6\sqrt{2},\quad t^\sigma=t_P(\tau_2)=8-6\sqrt{2}.\]
	The curve $\mathcal{F}_t$ is isomorphic over $K =\Q(\sqrt{2})$ to $\mathcal{F}_{t^\sigma}$ and can be identified as the curve \href{https://www.lmfdb.org/EllipticCurve/2.2.8.1/32.1/a/4}{\texttt{2.2.8.1-32.1-a4}} in LMFDB. We know from the LMFDB that $\mathcal{F}_t$ is the base change of the curves $\tilde{E}$: \href{https://www.lmfdb.org/EllipticCurve/Q/32/a/3}{\texttt{32.a3}} and $\tilde{E}^d$: \href{https://www.lmfdb.org/EllipticCurve/Q/64/a/3}{\texttt{64.a3}} over $\Q$. Hence by Proposition \ref{P:Lsplit}, we have
	\begin{equation}\label{eq:no_1_LE}
		L(\mathcal{F}_t,s)=L(\tilde{E},s)L(\tilde{E}^d,s)=L_\texttt{32aa1}(s)L_\texttt{64aa1}(s).
	\end{equation}
	Since $64$ is the smallest positive integer $N$ such that
	\[N\begin{pmatrix} 2 & 0 \\ 0 & 32 \end{pmatrix}^{-1} \text{ and } N\begin{pmatrix} 10 & 16 \\ 16 & 32 \end{pmatrix}^{-1}\]
	are even matrices and the Nebentypus $\left(\frac{16}{\cdot}\right)$ is trivial, we know that
	\begin{align*}
		\Theta_{P,\tau_1}(\tau) & = 32 q+64 q^2+64 q^5-96 q^9-128 q^{10}-192q^{13}+\cdots,  \\
		\Theta_{P,\tau_2}(\tau) & =32 q-64 q^2+64 q^5-96 q^9+128 q^{10}-192q^{13}+\cdots 
	\end{align*}
	are in $\mathcal{S}_2\left(\Gamma_0(64)\right)$ (see Proposition~\ref{P:MMasLcusp}). By using the fact that $\mathcal{M}_2\left(\Gamma_0(64)\right)$ has Sturm bound $16$ (see the character orbit space \href{https://www.lmfdb.org/ModularForm/GL2/Q/holomorphic/64/2/a/}{\texttt{64.2.a}}), one can easily verify that
	\[\Theta_{P,\tau_1}(\tau)=64f_\texttt{32aa1}(2\tau)+32f_\texttt{64aa1}(\tau),\quad \Theta_{P,\tau_2}(\tau)=-64f_\texttt{32aa1}(2\tau)+32f_\texttt{64aa1}(\tau).\]
	Thus, according to \eqref{eq:muasLcusp}, we re-prove \eqref{eq:no.1_MM_as_Lcusp}. To apply \textsf{PSLQ} to detect an identity of the form \eqref{eq:2x2detected}, we first evaluate $\frac{1}{\pi^4}L(\mathcal{F}_t,2)$ using \textsc{PARI/GP}. For example, the script
	\begin{verbatim}
			\pb 180 \\ Set the absolute error to 1/2^180 < 1/10^50
			K = nfinit(a^2 - 2);
			t = 8 + 6*a;
			E = ellinit([0, t*(t - 8), 0, 16*t^2, 0], K);
			lfun(E, 2)/Pi^4 \end{verbatim}
 	gives
	\[\frac{1}{\pi^4}L(\mathcal{F}_t,2)\approx 0.0096323473958955184559425659563904945672448438650884.\]
	On the other hand, we can use the hypergeometric formulas for $\mu(t)$ (see \cite[Eq.~(2-35)]{LR07}) to obtain the following numerical values:
	\begin{align*}
		\mu(t)^2 & \approx 1.4423982917768556815965690001765551621082444171345, \\
		\mu(t^\sigma)^2 & \approx 0.20945782510222931923592055775857185750090440240319.
	\end{align*}
	\textsf{PSLQ} then detects the identity
	\begin{equation*}
		\mu(t)^2-\mu(t^\sigma)^2\overset{?}{=}\frac{128}{\pi^4}L(\mathcal{F}_{t},2)=\frac{1}{2}L''(\mathcal{F}_{t},0),
	\end{equation*}
where the latter equality follows from the functional equation \eqref{E:FE}. The first equality can be verified immediately by using \eqref{eq:no.1_MM_as_Lcusp} and \eqref{eq:no_1_LE}.
\end{example}

\begin{example}[$\#11$]\label{ex:2x2_no.11}
	Let $\tau_1=[7,7,2],\tau_2=[1,-1,2]$. Then we have
	\[t=t_P(\tau_1)=-2032+768 \sqrt{7},\quad t^\sigma=t_P(\tau_2)=-2032-768 \sqrt{7}.\]
	The curves $\mathcal{F}_t$ and $\mathcal{F}_{t^\sigma}$ over $K=\Q(\sqrt{7})$ are in the isogeny class \href{https://www.lmfdb.org/EllipticCurve/2.2.28.1/1.1/a/}{\texttt{2.2.28.1-1.1-a}}. Since no curve in this isogeny class is the base change of a curve over $\Q$, the Weil restriction $A=\Res_{K/\Q}(\mathcal{F}_t)$ must be a simple abelian variety of dimension $2$ over $\Q$. It then follows from Proposition \ref{P:Lsplit} (i), there exists a CM newform $f$ with quadratic coefficient field $M$ such that
	\[L(\mathcal{F}_t,s)=L(f,s)L(f^\sigma,s),\]
	where $\sigma$ is the nontrivial element in $\Gal(M/\Q)$. Indeed, according to LMFDB, the $L$-function of $\mathcal{F}_t$, which is \href{https://www.lmfdb.org/L/4/28e2/1.1/c1e2/0/1}{\texttt{4-28e2-1.1-c1e2-0-1}}, can be factorized as
	\[L(\mathcal{F}_t,s)=L_\texttt{28da1}(s)L_\texttt{28da2}(s).\]
	Note that in this case, $\Theta_{P,\tau_1}$ and $\Theta_{P,\tau_2}$ are in the character orbit space \href{https://www.lmfdb.org/ModularForm/GL2/Q/holomorphic/112/2/f/}{\texttt{112.2.f}} with Sturm bound $32$. Hence we can verify that
	\begin{align*}
		\Theta_{P,\tau_1}(\tau)&=-i \sqrt{7}f_{\texttt{28da1}}(\tau)+i \sqrt{7}f_{\texttt{28da2}}(\tau)+\frac{7-i \sqrt{7}}{2}f_{\texttt{28da1}}(2\tau)+\frac{7+i \sqrt{7}}{2}f_{\texttt{28da2}}(2\tau),  \\
		\Theta_{P,\tau_2}(\tau)&=f_{\texttt{28da1}}(\tau)+f_{\texttt{28da2}}(\tau)+\frac{1+i\sqrt{7}}{2}f_{\texttt{28da1}}(2\tau)+\frac{1-i\sqrt{7}}{2}f_{\texttt{28da2}}(2\tau).
	\end{align*}
	Therefore, by \eqref{eq:muasLcusp}, we have
	\begin{align*}
		\mu(t) & =\frac{7}{\pi^2}((\sqrt{7}-9i)L_{\texttt{28da1}}(2)+(\sqrt{7}+9i)L_{\texttt{28da2}}(2)),  \\
		\mu(t^\sigma) & = \frac{1}{\pi^2}((9\sqrt{7}+7i)L_{\texttt{28da1}}(2)+(9\sqrt{7}-7i)L_{\texttt{28da2}}(2)),
	\end{align*}
	and the identity
	\[\mu(t)^2+7\mu(t^\sigma)^2=\frac{17248}{\pi^4}L(\mathcal{F}_t,2)=176L''(\mathcal{F}_t,0)\]
	(detected by \textsf{PSLQ}) can be verified at once.
\end{example}

\begin{example}[$\#15$]
	Let $\tau_1=[1,0,1],\tau_2=[2,2,1]$. Then we have
	\[t=t_Q(\tau_1)=270+162 \sqrt{3},\quad t^\sigma=t_Q(\tau_2)=270-162 \sqrt{3},\]
	both are not in $(-1,27)$.
	The curves $\mathcal{C}_t$ and $\mathcal{C}_{t^\sigma}$ are isogenous over $K=\Q(\sqrt{3})$. In this case, the isogeny class is \href{https://www.lmfdb.org/EllipticCurve/2.2.12.1/9.1/a/}{\texttt{2.2.12.1-9.1-a}} and $A=\Res_{K/\Q}(\mathcal{C}_t)$ is also simple. According to the LMFDB, we can determine that
	\[L(\mathcal{C}_t,s)=L_\texttt{36ba1}(s)L_\texttt{36ba2}(s).\]
	By using the fact that the character orbit space \href{https://www.lmfdb.org/ModularForm/GL2/Q/holomorphic/36/2/b/}{\texttt{36.2.b}} has Sturm bound $12$, one can verify that
	\[\Theta_{Q,\tau_1}(\tau)=2(f_{\texttt{36ba1}}(\tau)+f_{\texttt{36ba2}}(\tau)),\quad \Theta_{Q,\tau_2}(\tau)=2i\sqrt{2}(f_{\texttt{36ba1}}(\tau)-f_{\texttt{36ba2}}(\tau)).\]
	Therefore, by \eqref{eq:nuasLcusp}, we have
	\[\nu(t)=\frac{27\sqrt{3}}{4\pi^2}(L_{\texttt{36ba1}}(2)+L_{\texttt{36ba2}}(2)),\quad \nu(t^\sigma)=\frac{27i\sqrt{6}}{4\pi^2}(L_{\texttt{36ba1}}(2)-L_{\texttt{36ba2}}(2)),\]
	and the identity
	\[2\nu(t)^2+\nu(t^\sigma)^2=\frac{2187}{2\pi^4}L(\mathcal{C}_t,2)=\frac{27}{4}L''(\mathcal{C}_t,0)\]
	can be verified immediately.
\end{example}

\begin{example}[$\#23$] \label{ex:2x2_no.23}
	Let $\tau_1=[3,-3,13],\tau_2=[39,3,1]$. Then we have
	\[t=t_Q(\tau_1)=-163296-35640 \sqrt{21},\quad t^\sigma=t_Q(\tau_2)=-163296+35640 \sqrt{21}.\]
	Note that $\nu(t^\sigma)=\tilde{\nn}(t^\sigma)$ since $t^\sigma=26.99776\cdots$ is in $(-1,27)$. The curves $\mathcal{C}_t$ and $\mathcal{C}_{t^\sigma}$ over $K=\Q(\sqrt{21})$ are in the isogeny class \href{https://www.lmfdb.org/EllipticCurve/2.2.21.1/441.1/d/}{\texttt{2.2.21.1-441.1-d}}. Although neither $\mathcal{C}_t$ nor $\mathcal{C}_{t^\sigma}$ is the base change of a curve over $\Q$, the curve \href{https://www.lmfdb.org/EllipticCurve/2.2.21.1/441.1/d/6}{\texttt{2.2.21.1-441.1-d6}} in the same isogeny class is the base change of the curves $\tilde{E}$: \href{https://www.lmfdb.org/EllipticCurve/Q/441/d/2}{\texttt{441.d2}} and $\tilde{E}^d$: \href{https://www.lmfdb.org/EllipticCurve/Q/441/e/2}{\texttt{441.e2}} over $\Q$. Hence, we have
	\[L(\mathcal{C}_t,s)=L(\tilde{E},s)L(\tilde{E}^d,s)=L_\texttt{441ad1}(s)L_\texttt{441ae1}(s).\]
	One can verify that
	\[\Theta_{Q,\tau_1}(\tau)=6(f_{\texttt{441ad1}}(\tau)+f_{\texttt{441ae1}}(\tau)),\quad \Theta_{Q,\tau_2}(\tau)=21(-f_{\texttt{441ad1}}(\tau)+f_{\texttt{441ae1}}(\tau))\]
	by using the fact that the Sturm bound for $\mathcal{M}_2\left(\Gamma_0(441)\right)$ is $112$. (Note that the LMFDB only provides the first $100$ Fourier coefficients for each form. As an alternative, we can use the \textsc{PARI/GP} command \texttt{mfcoefs} to compute more coefficients to meet the Sturm bound here.) Therefore, by \eqref{eq:nuasLcusp}, we have
	\[\nu(t)=\frac{189}{8\pi^2}(L_{\texttt{441ad1}}(2)+L_{\texttt{441ae1}}(2)),\quad \nu(t^\sigma)=\frac{1323}{8\pi^2}(L_{\texttt{441ad1}}(2)-L_{\texttt{441ae1}}(2)),\]
	and
	\[49\nu(t)^2-\nu(t^\sigma)^2=\frac{1750329}{16\pi^4}L(\mathcal{C}_t,2)=-\frac{9}{2}L''(\mathcal{C}_t,0)\]
	can be verified immediately.
\end{example}

\begin{example}[$\#28$]\label{ex:4x4_no.28}
	Let $\tau_1=[16,16,7],\tau_2=[48,48,13],\tau_3=[48,0,1],\tau_4=[16,0,3]$. Then we have $t_P(\tau_i)=t_i,i=1,2,3,4$ as $\#28$ in Table \ref{quarticGalorbs_table}. The curves $\mathcal{F}_{t_i}$ are all defined over $K=\Q(\sqrt{2},\sqrt{3})$ and are in the isogeny class \href{https://www.lmfdb.org/EllipticCurve/4.4.2304.1/16.1/b/}{\texttt{4.4.2304.1-16.1-b}}. Note that the curve \href{https://www.lmfdb.org/EllipticCurve/4.4.2304.1/16.1/b/2}{\texttt{4.4.2304.1-16.1-b2}} in this isogeny class is the base change of the curves $\tilde{E}_1$: \href{https://www.lmfdb.org/EllipticCurve/2.2.12.1/16.1/a/8}{\texttt{2.2.12.1-16.1-a8}} and $\tilde{E}_2$: \href{https://www.lmfdb.org/EllipticCurve/2.2.12.1/256.1/c/8}{\texttt{2.2.12.1-256.1-c8}} over $\Q(\sqrt{3})$. By Proposition \ref{P:Lsplit}, we have
	\begin{equation}\label{eq:no_28_LE}
		L(\mathcal{F}_{t_i},s)=L(\tilde{E}_1,s)L(\tilde{E}_2,s)=L_\texttt{48ca1}(s)L_\texttt{48ca2}(s)L_\texttt{192ca1}(s)L_\texttt{192ca2}(s).
	\end{equation}
	Since the Sturm bound for the character orbit space \href{https://www.lmfdb.org/ModularForm/GL2/Q/holomorphic/192/2/c/}{\texttt{192.2.c}} is $64$, one can verify that
	\begin{align*}
		\Theta_{P,\tau_1}(\tau) & = -32(1-i\sqrt{3})f_\texttt{48ca1}(4\tau)-32(1+i\sqrt{3})f_\texttt{48ca2}(4\tau) \\
		& \quad +16(f_\texttt{192ca1}(\tau)+f_\texttt{192ca2}(\tau)), \\
		\Theta_{P,\tau_2}(\tau) & = -32(3+i\sqrt{3})f_\texttt{48ca1}(4\tau)-32(3-i\sqrt{3})f_\texttt{48ca2}(4\tau) \\
		& \quad +16 i \sqrt{3}(f_\texttt{192ca1}(\tau)-f_\texttt{192ca2}(\tau)), \\
		\Theta_{P,\tau_3}(\tau) & = 32(3+i\sqrt{3})f_\texttt{48ca1}(4\tau)+32(3-i\sqrt{3})f_\texttt{48ca2}(4\tau) \\
		& \quad +16 i \sqrt{3}(f_\texttt{192ca1}(\tau)-f_\texttt{192ca2}(\tau)), \\
		\Theta_{P,\tau_4}(\tau) & = 32(1-i\sqrt{3})f_\texttt{48ca1}(4\tau)+32(1+i\sqrt{3})f_\texttt{48ca2}(4\tau) \\
		& \quad +16(f_\texttt{192ca1}(\tau)+f_\texttt{192ca2}(\tau)).
	\end{align*}
	Therefore, by \eqref{eq:muasLcusp}, we have
	\begin{equation}\label{eq:no.28_MM_as_Lcusp}
		\begin{split}
			\mu(t_1) & = -\frac{1}{2\pi^2}((\sqrt{3}\!-\!3i)L_\texttt{48ca1}(2)\!+\!(\sqrt{3}\!+\!3i)L_\texttt{48ca2}(2)\!-\!8\sqrt{3}(L_\texttt{192ca1}(2)\!+\!L_\texttt{192ca2}(2))),\\
			\mu(t_2) & = -\frac{3}{2\pi^2}((\sqrt{3}+i)L_\texttt{48ca1}(2)+(\sqrt{3}-i)L_\texttt{48ca2}(2)-8i(L_\texttt{192ca1}(2)-L_\texttt{192ca2}(2))),\\
			\mu(t_3) & = \frac{3}{2\pi^2}((\sqrt{3}+i)L_\texttt{48ca1}(2)+(\sqrt{3}-i)L_\texttt{48ca2}(2)+8i(L_\texttt{192ca1}(2)-L_\texttt{192ca2}(2))),\\
			\mu(t_4) & = \frac{1}{2\pi^2}((\sqrt{3}\!-\!3i)L_\texttt{48ca1}(2)\!+\!(\sqrt{3}\!+\!3i)L_\texttt{48ca2}(2)\!+\!8\sqrt{3}(L_\texttt{192ca1}(2)\!+\!L_\texttt{192ca2}(2))).
		\end{split}
	\end{equation}
	Let $\m_i=\mu(t_i)$. To apply \textsf{PSLQ} to detect an identity of the form \eqref{eq:4x4regulator_expanded}, we need to compute $L(\mathcal{F}_{t_i},2)$ numerically with high precision. As indicated in the (sketched) proof of Theorem~\ref{T:4x4det}, by \eqref{eq:no_28_LE}, we can instead use the \textsc{PARI/GP} script
	\begin{verbatim}
			\pb 850 \\ Set the absolute error to 1/2^850 < 1/10^250
			mf1 = mfinit([48, 2, Mod(47, 48)], 0);
			mf2 = mfinit([192, 2, Mod(191, 192)], 0);
			lf1 = mfeigenbasis(mf1); lf2 = mfeigenbasis(mf2);
			f1 = lf1[1]; \\ Newform orbit 48.2.c.a (the unique orbit in 48.2.c)
			f2 = lf2[1]; \\ Newform orbit 192.2.c.a, the index may be different
			[L1, L2] = lfunmf(mf1, f1); [L3, L4] = lfunmf(mf2, f2);
			lfun(L1, 2)*lfun(L2, 2)*lfun(L3, 2)*lfun(L4, 2) \end{verbatim}
	to compute $L_\texttt{48ca1}(2)L_\texttt{48ca2}(2)L_\texttt{192ca1}(2)L_\texttt{192ca2}(2)$ to more than $250$ decimal places. Then, with the working precision set to $100$, \textsf{PSLQ} suggests that \eqref{eq:4x4regulator_expanded} seems to hold with
	\begin{align*}
		& a_{1}=9,\ a_{2}=1,\ a_{3}=1,\ a_{4}=9,\ a_{12}=6,\ a_{13}=6,\ a_{14}=-18,  \\
		& a_{23}=-2,\ a_{24}=6,\ a_{34}=6,\ a_{1234}=-24,\ r=331776.
   \end{align*}
   One can verify it by using \eqref{eq:no_28_LE} and \eqref{eq:no.28_MM_as_Lcusp}.
\end{example}

\begin{example}[$\#50$]\label{ex:4x4_no.50}
	Let $\tau_1\!=\![4, -4, 59],\tau_2\!=\![116, 116, 31],\tau_3\!=\![124, 116, 29],\tau_4\!=\![236, 4, 1]$. Then we have $t_P(\tau_i)=t_i,i=1,2,3,4$ as $\#50$ in Table \ref{quarticGalorbs_table}. The curves $\mathcal{F}_{t_i}$ are isogenous to each other over $K=\Q(\sqrt{2},\sqrt{29})$. It can be checked using \textsc{SageMath} that none of the curves in the isogeny class of $\mathcal{F}_{t_i}$ is the base change of an elliptic curve over a proper subfield of $K$, so for any $t\in\{t_1,t_t,t_3,t_4\}$, we have that $A=\Res_{K/\QQ}(\mathcal{F}_t)$ is simple. By using Milne's formula, we also have that
	\[\mathcal{N}_\QQ(A)=\Norm(\mathcal{N}_K(\mathcal{F}_t))d_{K}^2=189859297755136=3712^4.\]
	Thus, we know from the proof of Proposition \ref{P:Lsplit} that there exists a CM newform $f\in\mathcal{S}_2(\Gamma_1(3712))$ with some quartic number field $M$ as the coefficient field and
	\[L(\mathcal{F}_{t_i},s)=\prod_{\sigma\in \Gal(M/\QQ)}L(f^\sigma,s).\]
	Since the CM field of $\mathcal{F}_{t_i}$ is $F=\Q(\sqrt{-58})$,  we know from CM theory that $f$ has CM by $F$. Then, by results of Ribet \cite[Theorem~4.5]{Ribet77}, there is a Gr\"{o}ssencharacter $\psi$ of $F$ modulo $\mathfrak{m}$ such that
	\[f(\tau)=\sum_{\mathfrak{a} \text{ integral}}\psi(\mathfrak{a})q^{N(\mathfrak{a})}\in \mathcal{S}_2(\Gamma_0(|d_F|N(\mathfrak{m})),\chi_F\chi_{\psi}),\]
where $d_F$ is the discriminant of $F$, $\chi_F$ is the quadratic character associated to $F$, and $\chi_{\psi}$ is the Dirichlet character modulo $N(\mathfrak{m})$ given by $n\mapsto\psi((n))/n$. In this example, we have $d_F=-232$, $\chi_F=\left(\frac{-232}{\cdot}\right)$, and $|d_F|N(\mathfrak{m})=3712$. Since $\chi_F$ has odd parity and no weight $2$ modular forms exist for odd Nebentypus characters, we know immediately that $\chi_{\psi}$ is an odd Dirichlet character modulo $16$. By checking the LMFDB, we find that the only odd Dirichlet characters modulo $16$ are \href{https://www.lmfdb.org/Character/Dirichlet/16/c}{\texttt{16.c}}, \href{https://www.lmfdb.org/Character/Dirichlet/16/d}{\texttt{16.d}}, and \href{https://www.lmfdb.org/Character/Dirichlet/16/f}{\texttt{16.f}}, which correspond to the level $3712$ modular form character orbit spaces \href{https://www.lmfdb.org/ModularForm/GL2/Q/holomorphic/3712/2/g/}{\texttt{3712.2.g}}, \href{https://www.lmfdb.org/ModularForm/GL2/Q/holomorphic/3712/2/e/}{\texttt{3712.2.e}}, and \href{https://www.lmfdb.org/ModularForm/GL2/Q/holomorphic/3712/2/m/}{\texttt{3712.2.m}}, respectively. Since the newforms in these spaces have not yet been added to the LMFDB at present, we can use the script
	\begin{verbatim}
			\\ The default maximum stack size may need to be enlarged.
			mfg = mfinit([3712, 2, Mod(1217, 3712)], 0);
			mfe = mfinit([3712, 2, Mod(3073, 3712)], 0);
			mfm = mfinit([3712, 2, Mod(289, 3712)], 0);
			lfg = mfeigenbasis(mfg) \\ Compute newforms in 3712.2.g
			lfe = mfeigenbasis(mfe) \\ Compute newforms in 3712.2.e
			lfm = mfeigenbasis(mfm) \\ Compute newforms in 3712.2.m\end{verbatim}
to compute the newform orbits, which takes about $6$ minutes on a regular desktop (Intel Core i5-1135G7 Processor, 2.40GHz, 16GB RAM). From the output, we know that there are $9,8$, and $6$ newform orbits in \href{https://www.lmfdb.org/ModularForm/GL2/Q/holomorphic/3712/2/g/}{\texttt{3712.2.g}}, \href{https://www.lmfdb.org/ModularForm/GL2/Q/holomorphic/3712/2/e/}{\texttt{3712.2.e}}, and \href{https://www.lmfdb.org/ModularForm/GL2/Q/holomorphic/3712/2/m/}{\texttt{3712.2.m}}, respectively. Among these orbits, there is only one orbit in \href{https://www.lmfdb.org/ModularForm/GL2/Q/holomorphic/3712/2/g/}{\texttt{3712.2.g}} with dimension $4$, which consists of the following embeddings:
	\begin{align*}
		f_{\texttt{3712g1}}(\tau) & = q-3 q^9+5 q^{25}+\sqrt{29} q^{29}+2 i \sqrt{2} q^{31}-2 \sqrt{29} q^{37}+\cdots, \\
		f_{\texttt{3712g2}}(\tau) & = q-3 q^9+5 q^{25}+\sqrt{29} q^{29}-2 i \sqrt{2} q^{31}-2 \sqrt{29} q^{37}+\cdots, \\
		f_{\texttt{3712g3}}(\tau) & = q-3 q^9+5 q^{25}-\sqrt{29} q^{29}-2 i \sqrt{2} q^{31}+2 \sqrt{29} q^{37}+\cdots, \\
		f_{\texttt{3712g4}}(\tau) & = q-3 q^9+5 q^{25}-\sqrt{29} q^{29}+2 i \sqrt{2} q^{31}+2 \sqrt{29} q^{37}+\cdots.
	\end{align*}
	Hence it appears that
	\begin{equation}\label{eq:L_as_new_50}
		L(\mathcal{F}_{t_i},s)=L_\texttt{3712g1}(s)L_\texttt{3712g2}(s)L_\texttt{3712g3}(s)L_\texttt{3712g4}(s),
	\end{equation}
	where $L_{\texttt{3712g}i}(s):=L(f_{\texttt{3712g}i},s)$. We can again justify \eqref{eq:L_as_new_50}	by comparing the first few Dirichlet coefficients of both sides. By using the fact that \href{https://www.lmfdb.org/ModularForm/GL2/Q/holomorphic/3712/2/g/}{\texttt{3712.2.g}} has Sturm bound $960$, one can verify that
	\begin{align*}
		\Theta_{P,\tau_1}(\tau) & = 2(f_\texttt{3712g1}(\tau)+f_\texttt{3712g2}(\tau)+f_\texttt{3712g3}(\tau)+f_\texttt{3712g4}(\tau)), \\
		\Theta_{P,\tau_2}(\tau) & = 2\sqrt{29}(f_\texttt{3712g1}(\tau)+f_\texttt{3712g2}(\tau)-f_\texttt{3712g3}(\tau)-f_\texttt{3712g4}(\tau)), \\
		\Theta_{P,\tau_3}(\tau) & = -2i\sqrt{2}(f_\texttt{3712g1}(\tau)-f_\texttt{3712g2}(\tau)-f_\texttt{3712g3}(\tau)+f_\texttt{3712g4}(\tau)), \\
		\Theta_{P,\tau_4}(\tau) & = -2i\sqrt{58}(f_\texttt{3712g1}(\tau)-f_\texttt{3712g2}(\tau)+f_\texttt{3712g3}(\tau)-f_\texttt{3712g4}(\tau)).
	\end{align*}
	Therefore, according to \eqref{eq:muasLcusp}, we have
	\begin{equation}\label{eq:no.50_MM_as_Lcusp}
		\begin{split}
			\mu(t_1) & = \frac{4\sqrt{58}}{\pi^2}(L_\texttt{3712g1}(2)+L_\texttt{3712g2}(2)+L_\texttt{3712g3}(2)+L_\texttt{3712g4}(2)),\\
			\mu(t_2) & = \frac{116\sqrt{2}}{\pi^2}(L_\texttt{3712g1}(2)+L_\texttt{3712g2}(2)-L_\texttt{3712g3}(2)-L_\texttt{3712g4}(2)),\\
			\mu(t_3) & = -\frac{8i\sqrt{29}}{\pi^2}(L_\texttt{3712g1}(2)-L_\texttt{3712g2}(2)-L_\texttt{3712g3}(2)+L_\texttt{3712g4}(2)),\\
			\mu(t_4) & = -\frac{232i}{\pi^2}(L_\texttt{3712g1}(2)-L_\texttt{3712g2}(2)+L_\texttt{3712g3}(2)-L_\texttt{3712g4}(2)).
		\end{split}
	\end{equation}
	We can further use the script
	\begin{verbatim}
			\pb 850 \\ Set the absolute error to 1/2^850 < 1/10^250
			f = lfg[3]; \\ Select the newform orbit, the index may be different
			[L1, L2, L3, L4] = lfunmf(mfg, f);
			lfun(L1, 2)*lfun(L2, 2)*lfun(L3, 2)*lfun(L4, 2)\end{verbatim}
	to compute $L(\mathcal{F}_{t_i},2)$. Let $\m_i=\mu(t_i)$. \textsf{PSLQ} then detects that \eqref{eq:4x4regulator_expanded} seems to hold with
	\begin{align*}
		& a_{1}=3364,\ a_{2}=4,\ a_{3}=841,\ a_{4}=1,\ a_{12}=-232,\ a_{13}=3364,\ a_{14}=116,  \\
		& a_{23}=116,\ a_{24}=4,\ a_{34}=-58,\ a_{1234}=-464,\ r=741637881856.
   \end{align*}
   This can be verified by using \eqref{eq:L_as_new_50} and \eqref{eq:no.50_MM_as_Lcusp}.
\end{example}

\begin{example}[$\#54$]\label{ex:4x4_no.54}
	Let $\tau_1=[12,-12,7],\tau_2=[21,12,4],\tau_3=[48,0,1],\tau_4=[3,0,16]$. Then we have $t_Q(\tau_i)=t_i,i=1,2,3,4$ as $\#54$ in Table \ref{quarticGalorbs_table}. The curves $\mathcal{C}_{t_i}$ are all defined over $K=\Q(\sqrt{2},\sqrt{3})$ and are in the isogeny class \href{https://www.lmfdb.org/EllipticCurve/4.4.2304.1/324.1/a/}{\texttt{4.4.2304.1-324.1-a}}. Note that the curve \href{https://www.lmfdb.org/EllipticCurve/4.4.2304.1/324.1/a/7}{\texttt{4.4.2304.1-324.1-a7}} in this isogeny class is the base change of the curves $\tilde{E}_1$: \href{https://www.lmfdb.org/EllipticCurve/Q/36/a/4}{\texttt{36.a4}}, $\tilde{E}_2$: \href{https://www.lmfdb.org/EllipticCurve/Q/144/a/4}{\texttt{144.a4}}, $\tilde{E}_3$: \href{https://www.lmfdb.org/EllipticCurve/Q/576/e/4}{\texttt{576.e4}}, and $\tilde{E}_4$: \href{https://www.lmfdb.org/EllipticCurve/Q/576/f/4}{\texttt{576.f4}} over $\Q$. By Proposition \ref{P:Lsplit}, we have
	\begin{equation}\label{eq:L_as_new_54}
		L(\mathcal{C}_{t_i},s)=\prod_{i=1}^{4}L(\tilde{E}_i,s)=L_\texttt{36aa1}(s)L_\texttt{144aa1}(s)L_\texttt{576ae1}(s)L_\texttt{576af1}(s).	
	\end{equation}
	Since the Sturm bound for $\mathcal{M}_2\left(\Gamma_0(576)\right)$ is $192$, one can verify that
	\begin{align*}
		\Theta_{Q,\tau_1}(\tau) & = -6(f_\texttt{36aa1}(\tau)+f_\texttt{144aa1}(\tau)-f_\texttt{576ae1}(\tau)-f_\texttt{576af1}(\tau)) \\
		& \quad +24(f_\texttt{36aa1}(4\tau)-4f_\texttt{36aa1}(16\tau)+f_\texttt{144aa1}(4\tau)), \\
		\Theta_{Q,\tau_2}(\tau) & = -6(f_\texttt{36aa1}(\tau)-f_\texttt{144aa1}(\tau)+f_\texttt{576ae1}(\tau)-f_\texttt{576af1}(\tau)) \\
		& \quad +24(f_\texttt{36aa1}(4\tau)-4f_\texttt{36aa1}(16\tau)-f_\texttt{144aa1}(4\tau)), \\
		\Theta_{Q,\tau_3}(\tau) & = 12(f_\texttt{36aa1}(\tau)-f_\texttt{144aa1}(\tau)-f_\texttt{576ae1}(\tau)+f_\texttt{576af1}(\tau))\\
		& \quad -48(f_\texttt{36aa1}(4\tau)-4f_\texttt{36aa1}(16\tau)-f_\texttt{144aa1}(4\tau)), \\
		\Theta_{Q,\tau_4}(\tau) & = 3(f_\texttt{36aa1}(\tau)+f_\texttt{144aa1}(\tau)+f_\texttt{576ae1}(\tau)+f_\texttt{576af1}(\tau))\\
		& \quad -12(f_\texttt{36aa1}(4\tau)-4f_\texttt{36aa1}(16\tau)+f_\texttt{144aa1}(4\tau)).
	\end{align*}
	Therefore, by \eqref{eq:nuasLcusp}, we have
	\begin{equation}\label{eq:no.54_MM_as_Lcusp}
		\begin{split}
			\nu(t_1) & = -\frac{27}{16\pi^2}(13 L_\texttt{36aa1}(2) + 12 L_\texttt{144aa1}(2) - 16 (L_\texttt{576ae1}(2) + L_\texttt{576af1}(2))), \\
			\nu(t_2) & = \frac{27}{8\pi^2}(13 L_\texttt{36aa1}(2) - 12 L_\texttt{144aa1}(2) + 16 (L_\texttt{576ae1}(2) - L_\texttt{576af1}(2))), \\
			\nu(t_3) & = \frac{27}{8\pi^2}(13 L_\texttt{36aa1}(2) - 12 L_\texttt{144aa1}(2) -16 (L_\texttt{576ae1}(2) - L_\texttt{576af1}(2))), \\
			\nu(t_4) & = \frac{27}{32\pi^2}(13 L_\texttt{36aa1}(2) + 12 L_\texttt{144aa1}(2) + 16 (L_\texttt{576ae1}(2) + L_\texttt{576af1}(2))).
		\end{split}
	\end{equation}
	Let $\m_i=\nu(t_i)$. According to \textsf{PSLQ}, \eqref{eq:4x4regulator_expanded} seems to hold with
	\begin{align*}
		& a_{1}=16,\ a_{2}=1,\ a_{3}=1,\ a_{4}=256,\ a_{12}=-8,\ a_{13}=-8,\ a_{14}=-128,  \\
		& a_{23}=-2,\ a_{24}=-32,\ a_{34}=-32,\ a_{1234}=-64,\ r=1326476736.
   \end{align*}
   This can be verified by using \eqref{eq:L_as_new_54} and \eqref{eq:no.54_MM_as_Lcusp}.
\end{example}

\begin{example}[$\#71$]\label{ex:4x4_no.71}
	Let $\tau_1=[3,-3,47],\tau_2=[15,-15,13],\tau_3=[39, 15, 5],\tau_4=[141,3,1]$. Then we have $t_Q(\tau_i)=t_i,i=1,2,3,4$ as $\#71$ in Table \ref{quarticGalorbs_table}. The curves $\mathcal{C}_{t_i}$ are isogenous to each other over $K=\Q(\sqrt{5},\sqrt{37})$. Similar to Example~\ref{ex:4x4_no.50}, it can be checked that the elliptic curves in the isogeny class of $\mathcal{C}_{t_i}$ are not base changes, so for any $t\in\{t_1,t_t,t_3,t_4\}$, the abelian variety $A=\Res_{K/\QQ}(\mathcal{C}_t)$ is simple. Since $\mathcal{N}_\QQ(A)=7685231450625=1665^4$ by Milne's formula, we know that there exists a CM newform $f\in\mathcal{S}_2(\Gamma_1(1665))$ with some quartic coefficient field $M$ such that
	\[L(\mathcal{C}_{t_i},s)=\prod_{\sigma\in \Gal(M/\QQ)}L(f^\sigma,s).\]
	In this example, the CM field of $\mathcal{C}_{t_i}$ is $F=\Q(\sqrt{-555})$, whose discriminant is $-555$, so the Nebentypus attached to $f$ is $\left(\frac{-555}{\cdot}\right)\chi$, where $\chi$ a Dirichlet character modulo $3$. This implies that $\left(\frac{-555}{\cdot}\right)\chi$ must take values in $\Q$; i.e. it is a character of degree $1$. The data of all level $1665$ newform orbits with Nebentypus of degree $1$ are available in the LMFDB (see the modular form space \href{https://www.lmfdb.org/ModularForm/GL2/Q/holomorphic/1665/2/}{\texttt{1665.2}}) and \href{https://www.lmfdb.org/ModularForm/GL2/Q/holomorphic/1665/2/g/a/}{\texttt{1665.2.g.a}} is the unique orbit of dimension $4$ among them. We can thus immediately conclude that
	\begin{equation}\label{eq:L_as_new_71}
		L(\mathcal{C}_{t_i},s)=L_\texttt{1665ga1}(s)L_\texttt{1665ga2}(s)L_\texttt{1665ga3}(s)L_\texttt{1665ga4}(s).
	\end{equation}
By using the fact that the Sturm bound for the character orbit space \href{https://www.lmfdb.org/ModularForm/GL2/Q/holomorphic/1665/2/g/}{\texttt{1665.2.g}} is $456$, one can verify that
	\begin{align*}
		\Theta_{Q,\tau_1}(\tau) & = 3(f_\texttt{1665ga1}(\tau)+f_\texttt{1665ga2}(\tau)+f_\texttt{1665ga3}(\tau)+f_\texttt{1665ga4}(\tau)),  \\
		\Theta_{Q,\tau_2}(\tau) & = 3 i \sqrt{5}(f_\texttt{1665ga1}(\tau)+f_\texttt{1665ga2}(\tau)-f_\texttt{1665ga3}(\tau)-f_\texttt{1665ga4}(\tau)), \\
		\Theta_{Q,\tau_3}(\tau) & = -\frac{3 \sqrt{37}}{2}(f_\texttt{1665ga1}(\tau)-f_\texttt{1665ga2}(\tau)+f_\texttt{1665ga3}(\tau)-f_\texttt{1665ga4}(\tau)), \\
		\Theta_{Q,\tau_4}(\tau) & = -\frac{3 i \sqrt{185}}{2}(f_\texttt{1665ga1}(\tau)-f_\texttt{1665ga2}(\tau)-f_\texttt{1665ga3}(\tau)+f_\texttt{1665ga4}(\tau)).
	\end{align*}
	Therefore, according to \eqref{eq:nuasLcusp}, we have
	\begin{equation}\label{eq:no.71_MM_as_Lcusp}
		\begin{split}
			\nu(t_1) & = \frac{27\sqrt{185}}{16\pi^2}(L_\texttt{1665ga1}(2)+L_\texttt{1665ga2}(2)+L_\texttt{1665ga3}(2)+L_\texttt{1665ga4}(2)), \\
			\nu(t_2) & = \frac{135i\sqrt{37}}{16\pi^2}(L_\texttt{1665ga1}(2)+L_\texttt{1665ga2}(2)-L_\texttt{1665ga3}(2)-L_\texttt{1665ga4}(2)), \\
			\nu(t_3) & = \frac{999\sqrt{5}}{16\pi^2}(L_\texttt{1665ga1}(2)-L_\texttt{1665ga2}(2)+L_\texttt{1665ga3}(2)-L_\texttt{1665ga4}(2)), \\
			\nu(t_4) & = \frac{4995i}{16\pi^2}(L_\texttt{1665ga1}(2)-L_\texttt{1665ga2}(2)-L_\texttt{1665ga3}(2)+L_\texttt{1665ga4}(2)).
		\end{split}
	\end{equation}
	\textsf{PSLQ} detects that \eqref{eq:4x4regulator_expanded} seems to hold with
	\begin{align*}
	 	& a_{1}=8761600,\ a_{2}=350464,\ a_{3}=6400,\ a_{4}=256,\ a_{12}=3504640,\ a_{13}=-473600,  \\
	 	& a_{14}=94720 ,\  a_{23}=94720,\ a_{24}=-18944,\ a_{34}=2560,\ a_{1234}=-378880,\\
		& r=622503747500625.
	\end{align*}
	This can be verified by using \eqref{eq:L_as_new_71} and \eqref{eq:no.71_MM_as_Lcusp}.
\end{example}

\section{Closing remarks}\label{S:remarks}
We conclude this paper by giving a few remarks about our and some  related results from both theoretical and computational perspectives. Firstly, it is crucial to assume that the fields of definition of elliptic curves in Theorem~\ref{T:2x2det} and Theorem~\ref{T:4x4det} (and their potential generalizations) are totally real. This is clearly visible in the quadratic case. For example, let $K=\QQ(\sqrt{-3})$, $\Gal(K/\QQ)=\langle\sigma\rangle$, and $t=8+8i\sqrt{3}\in K$. Then the curve $\mathcal{F}_t$ can be identified as the curve \href{https://www.lmfdb.org/EllipticCurve/2.0.3.1/256.1/CMa/2}{\texttt{2.0.3.1-256.1-CMa2}} or \href{https://www.lmfdb.org/EllipticCurve/2.0.3.1/256.1/CMb/2}{\texttt{2.0.3.1-256.1-CMb2}} (depending on which embedding $K\xhookrightarrow{}\CC$ is used). According to the data in the LMFDB, its $L$-function is still the product of modular $L$-functions, namely $L(\mathcal{F}_t,s)=L_\textnormal{\texttt{48ca1}}(s)L_\textnormal{\texttt{48ca2}}(s)$. However, since $\mu(t)=\mu(\bar{t})=\mu(t^\sigma)$, the (non-zero) determinant of the form \eqref{eq:2x2regulator} must be a non-zero rational multiple of $\mu(t)^2$. By \eqref{eq:mu88} and \eqref{E:FE}, this quantity is certainly not a rational multiple of $L''(\mathcal{F}_t,0)$, so \eqref{2x2Gamma04} is invalid for this curve. Note, however, that this example does not invalidate the Beilinson's conjecture. To elaborate further, recall that every CM elliptic curve defined over a number field is a $\QQ$\textit{-curve}; i.e., it is isogenous (over $\overline{\QQ}$) to all its Galois conjugates. Although $\mathcal{F}_t$ has CM, $\mathcal{F}_t$ is not $K$-isogenous to $\mathcal{F}_{t^\sigma}$, so we cannot canonically construct (linearly independent) elements in $K_2^T(\mathcal{F}_t)$ using pullbacks of an isogeny. In other words, over the field $K$, the determinant \eqref{eq:2x2regulator} in this case may differ from the Beilinson regulator \eqref{eq:nxn_regulator} (for $n=2$). This discrepancy seems to carry over in higher dimensions; for all CM curves defined over non-totally real number fields of degree $d\le 4$ obtained from Proposition~\ref{P:CMPQ}, we are unable to find any simple identity between determinants of Mahler measures and their $L$-values.

Secondly, the reader may observe that there are no identities for $3\times3$ determinants in Section~\ref{S:results}, due to the absence of CM elliptic curves in the families $\mathcal{F}_t$ and $\mathcal{C}_t$ which are defined over totally real cubic fields. On the other hand, we would like to point out evidence of such identity from another family
\[R_t=(1+x)(1+y)(x+y)-txy,\]
whose Mahler measure was also originally studied by Boyd \cite{Boyd98}. Proven Mahler measure formulas for $R_t$ analogous to \eqref{E:mPt2} and \eqref{E:mQt2} are listed in \cite[Table~1]{MS20}. Consider $f(x)=x^3+x^2-2x-1$, which has three distinct roots
\[p_1=-1.8019\cdots,\  p_2=-0.4450\cdots,\ p_3=1.2469\cdots,\]
and let $K$ be the splitting field of $f(x)$, which is the cubic Galois field of smallest discriminant. For each $i=1,2,3$, let $t_i=2(1+p_i)^2/p_i$ and $\m_i=\m(R_{t_i})$. The curve $R_{t}=0$ can be written in a Weierstrass form 
\[E_t: Y^2=X^3+(t^2-4t-8)X^2+16(t+1)X\]
and for any $i=1,2,3$, $E_{t_i}$ can be identified as the curve  
\href{https://www.lmfdb.org/EllipticCurve/3.3.49.1/64.1/a/7}{\texttt{3.3.49.1-64.1-a7}} in LMFDB, which is the base change of an elliptic curve $E$ over $\QQ$ with label \href{https://www.lmfdb.org/EllipticCurve/Q/196/b/2}{\texttt{196.b2}}. Therefore, we have that $L(E_{t_i},s)=L(E,s)L(E\otimes \chi,s)L(E\otimes \chi^2,s),$ where $\chi$ is a character of order $3$. By modularity of $E$, this equation can then be rewritten as 
\begin{equation}\label{eq:Lcubic}
L(E_{t_i},s)=L(f_\texttt{28ea1},s)L(\bar{f}_\texttt{28ea1},s)L(f_\texttt{196ab1},s).
\end{equation}
 In a forthcoming paper \cite{BLW}, Brunault, Liu, and Wang introduce an algorithm based on Rogers-Zudilin method in order to prove certain Mahler measure identities for curves in the family $R_t=0$ defined over number fields. In particular, they give explicit expressions of $\m_1,\m_2,$ and $\m_3$ in terms of linear combinations of $\frac{1}{\pi^2}L(f_\texttt{28ea1},2)$, $\frac{1}{\pi^2}L(\bar{f}_\texttt{28ea1},2)$ and $\frac{1}{\pi^2}L(f_\texttt{196ab1},2)$. By using their result and \eqref{eq:Lcubic}, one can deduce easily that
\begin{equation}\label{eq:3x3det}
\det\begin{pmatrix}
\m_1 & \m_2 & \m_3\\
\m_3 & \m_1 & \m_2\\
\m_2 & \m_3 & \m_1
\end{pmatrix} =\frac{7203}{2\pi^6}L(E_{t_i},2)=\frac{1}{4}L^{(3)}(E_{t_i},0). 
\end{equation}
Note that the curves $E_{t_i}$ have no CM, but they are isogenous to each other over $K$; i.e., they are $\Q$-curves, and the second equality in \eqref{eq:3x3det} follows from the functional equation \eqref{E:FE} for elliptic curves over totally real fields, thanks to Taylor's potential modularity theorem \cite{Taylor02}. Furthermore, under certain assumptions, $\QQ$-curves are known to be strongly modular \cite[Theorem~5.3]{GQ14}. Together with this fact, the existence of \eqref{eq:3x3det} suggests that the results in this paper could be extended beyond the CM case, although more sophisticated techniques will be required when computing Mahler measures of non-CM curves. In particular, for the quadratic case, we formulate the following conjecture based on our numerical experiments.
\begin{conjecture}\label{conj:non-CM}
Let $K$ be a real quadratic field and $t\in K\backslash\mathbb{Q}$. If $\mathcal{F}_t$ (resp. $\mathcal{C}_t$) is a $\mathbb{Q}$-curve which is $K$-isogenous to $\mathcal{F}_{t^\sigma}$ (resp. $\mathcal{C}_{t^\sigma}$), where $\Gal(\Q(t)/\Q)=\langle \sigma \rangle$, then
there exist $a,b\in \Z\setminus\{0\},a>0$ and $r,s\in \Q$ such that
\eqref{2x2Gamma04} (resp. \eqref{2x2Gamma03}) holds.
\end{conjecture}
We collect the relevant data in Conjecture~\ref{conj:non-CM} for non-CM curves over $K=\mathbb{Q}(\sqrt{d})$ with $2\le d\le 17$ in Table~\ref{conj_for_non-CM_Q-curves}. These curves are initially classified using $j$-invariants of $\mathbb{Q}$-curves over quadratic fields stored in the LMFDB. Then we use $\textsc{SageMath}$ to filter out those which are not $K$-isogenous to their Galois conjugates.

\begin{remark}
Note that the conjectural identities for $\mathcal{F}_{8+9\sqrt{2}}, \mathcal{F}_{\frac{49}{2}+\frac{9\sqrt{17}}{2}}$, and $\mathcal{C}_{112+38 \sqrt{5}}$ in Table~\ref{conj_for_non-CM_Q-curves} are immediate consequences of the conjectural identities for $\mu(8\pm 9\sqrt{2}), \mu\left(\frac{49}{2}\pm\frac{9\sqrt{17}}{2}\right)$, and $\nn(112\pm 38 \sqrt{5})$ in \cite{Samart15}. The identities for $\mathcal{F}_{24+8\sqrt{5}}, \mathcal{F}_{\frac{11}{2}+\frac{3\sqrt{13}}{2}}$, and $\mathcal{C}_{20+14\sqrt{2}}$ in Table~\ref{conj_for_non-CM_Q-curves} can be proven using results in \cite{BLW}.
\end{remark}


\section*{Acknowledgements}
The authors would like to thank John Voight, Hang Liu, and François Brunault for helpful discussions, and the anonymous referees for valuable comments on earlier versions of this paper, which greatly helped improve the exposition. The authors are also grateful to Ken Ribet, Xavier Guitart, Edgar Costa, and David Farmer for answering their questions about CM elliptic curves and $L$-functions.  The first author is supported by the National Research Council of Thailand (NRCT) under the Research Grant for Mid-Career Scholar [N41A640153 to D.S.]. The second author is supported by the Natural Science Foundation of Anhui Province (Grant No.~2508085QA017).

\appendix
\section{Data tables}\label{A:tables}

\begin{longtable}{|c|c|c|c|c|c|c|}
	\caption{$2\times 2$ determinant identities \eqref{2x2Gamma04} and \eqref{2x2Gamma03}}
	\label{2x2identities_table}\\

	\hline

	$\#$ & $t$ & $t^\sigma$ & $a$ & $b$ & $r$ & $s$ \\

	\hline

	\endfirsthead

	\caption[]{(continued)}\\

	\hline

	$\#$ & $t$ & $t^\sigma$ & $a$ & $b$ & $r$ & $s$ \\

	\hline

	\endhead

	$1$ & $t_P([16,0,1])=8+6\sqrt{2}$ & $t_P([16,16,5])$ & $1$ & $-1$ & $128$ & {\small $\frac{1}{2}$} \\

	\hline

	$2$ & $t_P([8,0,1])=8+8\sqrt{2}$ & $t_P([8,8,3])$ & $1$ & $1$ & $256$ & {\small$\frac{1}{2}$} \\

	\hline

	$3$ & $t_P([12,4,1])=-32+32\sqrt{2}$ & $t_P([4,-4,3])$ & $1$ & $2$ & $1024$ & {\small$\frac{1}{2}$} \\

	\hline

	$4$ & $t_P([2,0,1])=48+32\sqrt{2}$ & $t_P([6,-4,1])$ & $1$ & $1$ & $512$ & $1$ \\

	\hline

	$5$ & $t_P([20,4,1])=-256+192\sqrt{2}$ & $t_P([4,-4,5])$ & $1$ & $-4$ & $-4096$ & {\small$\frac{1}{2}$} \\

	\hline

	$6$ & $t_P([1,0,1])=272+192\sqrt{2}$ & $t_P([5,4,1])$ & $1$ & $-1$ & $1024$ & $4$ \\

	\hline

	$7$ & $t_P([16,4,1])=8+4\sqrt{3}$ & $t_P([16,12,3])$ & $1$ & $3$  & $144$ & {\small$\frac{1}{2}$} \\

	\hline

	$8$ & $t_P([3,3,1])=-112+64 \sqrt{3}$ & $t_P([1,-1,1])$ & $1$ & $3$  & $2304$  & $8$ \\
\hline

\end{longtable}

{\tiny
\begin{longtable}{|c|c|}
	\caption{Galois orbits in totally real quartic fields}
	\label{quarticGalorbs_table}\\

	\hline

	$\#$ & $t_1,t_2,t_3,t_4$ with $t_1<t_2<t_3<t_4$ \\

	\hline

	\endfirsthead

	\caption[]{(continued)}\\

	\hline

	$\#$ & $t_1,t_2,t_3,t_4$ with $t_1<t_2<t_3<t_4$ \\

	\hline

	\endhead


	$28$ & $t_1\!=\!t_P([16,16,7])\!=\!8\!-\!3 \sqrt{2}\!-\!5 \sqrt{6}, t_2\!=\!t_P([48,48,13])\!=\!++-, t_3\!=\!t_P([48,0,1])\!=\!+-+, t_4\!=\!t_P([16,0,3])\!=\!+++$\\

	\hline


	$29$ & $t_1\!=\!t_P([8,-8,5])\!=\!8\!-\!16 \sqrt{3}\!-\!8 \sqrt{6},t_2\!=\!t_P([24,24,7])\!=\!+-+, t_3\!=\!t_P([24,0,1])\!=\!++-, t_4\!=\!t_P([8,0,3])\!=\!+++$\\

	\hline


	$30$ & $t_1\!=\!t_P([8,-8,11])\!=\!8\!-\!280\sqrt{2}\!-\!224\sqrt{3},t_2\!=\!t_P([72,72,19])\!=\!+\!-\!+, t_3\!=\!t_P([72,0,1])\!=\!+\!+\!-, t_4\!=\!t_P([8,0,9])\!=\!+\!+\!+$\\

	\hline

	$31$ & \makecell[c]{$t_1=t_P([4,-4,7])=32 (-17-12 \sqrt{2}-10 \sqrt{3}-7 \sqrt{6}),$\\$t_2=t_P([12,12,5])=32(-+-+), t_3=t_P([20,12,3])=32(-++-), t_4=t_P([28,4,1])=32(--++)$}\\

	\hline

	$32$ & \makecell[c]{$t_1=t_P([14,12,3])=16(35+24\sqrt{2}-20\sqrt{3}-14\sqrt{6}),$\\$t_2=t_P([10,4,1])=16(+--+), t_3=t_P([6,0,1])=16(+-+-), t_4=t_P([2,0,3])=16(++++)$}\\

	\hline

	$33$ & \makecell[c]{$t_1=t_P([4,-4,13])=64 (-208-147 \sqrt{2}-120 \sqrt{3}-85 \sqrt{6}),$\\$t_2=t_P([12,-12,7])=64(-++-), t_3=t_P([28,12,3])=64(--++), t_4=t_P([52,4,1])=64(-+-+)$}\\

	\hline

	$34$ & \makecell[c]{$t_1=t_P([13,12,3])=16(833-588\sqrt{2}+480\sqrt{3}-340\sqrt{6}),$\\$t_2=t_P([7,4,1])=16(++--),t_3=t_P([3,0,1])=16(+--+), t_4=t_P([1,0,3])=16(++++)$}\\

	\hline

	$35$ & \makecell[c]{$t_1=t_P([4,-4,19])=32 (-4801-3395 \sqrt{2}-2772 \sqrt{3}-1960 \sqrt{6}),$\\$ t_2=t_P([36,36,11])=32(--++), t_3=t_P([44,36,9])=32(-+-+), t_4=t_P([76,4,1])=32(-++-)$}\\
\hline
\end{longtable}
}

{
\begin{longtable}{|c|c|c|c|c|c|c|c?c|c|c|c|c|c|c|c|}
	\caption{$4\times 4$ determinant identities \eqref{det4x4Gamma04} and \eqref{det4x4Gamma03}}
	\label{4x4identities_table}\\

	\hline

	$\#$ & $j$ & $a$ & $b$ & $c$ & $d$ & $r$ & $s$ & $\#$ & $j$ & $a$ & $b$ & $c$ & $d$ & $r$ & $s$ \\

	\hline

	\endfirsthead

	\caption[]{(continued)}\\

	\hline

	$\#$ & $j$ & $a$ & $b$ & $c$ & $d$ & $r$ & $s$ & $\#$ & $j$ & $a$ & $b$ & $c$ & $d$ & $r$ & $s$ \\

	\hline

	\endhead

	$28$ & $5$ & $3$ & $1$ & $1$ & $3$ & {\footnotesize$331776$} & {\footnotesize$\frac{1}{24}$} & $52$ & $4$ & $2$ & $1$ & $1$ & $8$ & {\footnotesize $13856832$} & {\footnotesize $\frac{891}{32}$} \\

	\hline

	$29$ & $2$ & $3$ & $1$ & $1$ & $3$ & {\footnotesize$5308416$} & {\footnotesize$\frac{1}{24}$} & $53$ & $4$ & $2$ & $1$ & $1$ & $8$ & {\footnotesize $93533616$} & {\footnotesize $\frac{297}{128}$} \\

	\hline

	$30$ & $1$ & $9$ & $1$ & $1$ & $9$ & {\footnotesize$360972288$} & {\footnotesize$\frac{17}{6}$} & $54$ & $6$ & $4$ & $1$ & $1$ & $16$ & {\footnotesize $1326476736$} & {\footnotesize $-\frac{1053}{128}$} \\

	\hline

	$31$ & $3$ & $6$ & $2$ & $3$ & $1$ & {\footnotesize$84934656$} & {\footnotesize$\frac{1}{24}$} & $55$ & $4$ & $2$ & $1$ & $1$ & $8$ & {\footnotesize $863742528$} & {\footnotesize $\frac{55539}{32}$} \\

	\hline

	$32$ & $2$ & $3$ & $1$ & $1$ & $3$ & {\footnotesize$21233664$} & {\footnotesize$\frac{1}{6}$} & $56$ & $4$ & $1$ & $2$ & $5$ & $10$ & {\footnotesize $332150625$} & {\footnotesize $\frac{27}{128}$} \\

	\hline

	$33$ & $4$ & $12$ & $4$ & $3$ & $1$ & {\footnotesize$1358954496$} & {\footnotesize$\frac{1}{24}$} & $57$ & $2$ & $5$ & $1$ & $2$ & $10$ & {\footnotesize $\frac{12301875}{2}$} & {\footnotesize $\frac{81}{16}$} \\

	\hline

	$34$ & $5$ & $3$ & $1$ & $1$ & $3$ & {\footnotesize$84934656$} & {\footnotesize $\frac{32}{3}$} & $58$ & $2$ & $5$ & $1$ & $2$ & $10$ & {\footnotesize $\frac{1422916875}{2}$} & {\footnotesize $\frac{9369}{16}$} \\

	\hline

	$35$ & $1$ & $18$ & $2$ & $9$ & $1$ & {\footnotesize$6115295232$} & $3$ & $59$ & $1$ & $1$ & $4$ & $5$ & $20$ & {\footnotesize $\frac{27568501875}{8}$} & {\footnotesize $\frac{2241}{64}$} \\
\hline

\end{longtable}
}

{\small
\begin{longtable}{|cc?cc?cc|}
	\caption{$L(\mathcal{F}_{t},s)$ or $L(\mathcal{C}_{t},s)$ as products of $L$-functions of newforms}
	\label{L_as_products_table}\\

	\hline

	$\#$  & {\scriptsize $L(\mathcal{F}_{t},s)$ or $L(\mathcal{C}_{t},s)$} & $\#$  & {\scriptsize $L(\mathcal{F}_{t},s)$ or $L(\mathcal{C}_{t},s)$} & $\#$  & {\scriptsize $L(\mathcal{F}_{t},s)$ or $L(\mathcal{C}_{t},s)$} \\

	\hline

	\endfirsthead

	\caption[]{(continued)}\\

	\hline

	$\#$  & {\scriptsize $L(\mathcal{F}_{t},s)$ or $L(\mathcal{C}_{t},s)$} & $\#$  & {\scriptsize $L(\mathcal{F}_{t},s)$ or $L(\mathcal{C}_{t},s)$} & $\#$  & {\scriptsize $L(\mathcal{F}_{t},s)$ or $L(\mathcal{C}_{t},s)$} \\

	\hline

	\endhead

	\hline

	\endfoot

	\hline

	\endlastfoot

	$1$ & $L_\texttt{32aa1}L_\texttt{64aa1}$ & $26$ & $L_\texttt{369db1}L_\texttt{369db2}$ & $51$ & $L_\texttt{1856ga1}L_\texttt{1856ga2}L_\texttt{1856ga3}L_\texttt{1856ga4}$ \\

	$2$ & $L_\texttt{64ba1}L_\texttt{64ba2}$ & $27$ & $L_\texttt{801da1}L_\texttt{801da2}$ & $52$ & $L_\texttt{24fa1}L_\texttt{24fa2}L_\texttt{96fa1}L_\texttt{96fa2}$ \\

	$3$ & $L_\texttt{128ba1}L_\texttt{128ba2}$ & $28$ & $L_\texttt{48ca1}L_\texttt{48ca2}L_\texttt{192ca1}L_\texttt{192ca2}$ & $53$ & $L_\texttt{72da1}L_\texttt{72da2}L_\texttt{288da1}L_\texttt{288da2}$ \\

	$4$ & $L_\texttt{64ba1}L_\texttt{64ba2}$ & $29$ & $L_\texttt{192fa1}L_\texttt{192fa2}L_\texttt{192fa3}L_\texttt{192fa4}$ & $54$ & $L_\texttt{36aa1}L_\texttt{144aa1}L_\texttt{576ae1}L_\texttt{576af1}$ \\

	$5$ & $L_\texttt{256ab1}L_\texttt{256ac1}$ & $30$ & $L_\texttt{64ba1}L_\texttt{64ba2}L_\texttt{576da1}L_\texttt{576da2}$ & $55$ & $L_\texttt{24fa1}L_\texttt{24fa2}L_\texttt{96fa1}L_\texttt{96fa2}$ \\

	$6$ & $L_\texttt{32aa1}L_\texttt{64aa1}$ & $31$ & $L_\texttt{384fa1}L_\texttt{384fa2}L_\texttt{384fa3}L_\texttt{384fa4}$ & $56$ & $L_\texttt{360da1}L_\texttt{360da2}L_\texttt{360da3}L_\texttt{360da4}$ \\

	$7$ & $L_\texttt{48ca1}L_\texttt{48ca2}$ & $32$ & $L_\texttt{192fa1}L_\texttt{192fa2}L_\texttt{192fa3}L_\texttt{192fa4}$ & $57$ & $L_\texttt{60ha1}L_\texttt{60ha2}L_\texttt{60ha3}L_\texttt{60ha4}$ \\

	$8$ & $L_\texttt{48ca1}L_\texttt{48ca2}$ & $33$ & $L_\texttt{768cg1}L_\texttt{768cg2}L_\texttt{768cg3}L_\texttt{768cg4}$ & $58$ & $L_\texttt{60ha1}L_\texttt{60ha2}L_\texttt{60ha3}L_\texttt{60ha4}$ \\
\end{longtable}
}

{\scriptsize
\begin{longtable}{|cc?cc?cc?cc|}
	\caption{Newforms and their LMFDB labels}
	\label{newforms_table}\\

	\hline

	Newform & LMFDB & Newform & LMFDB & Newform & LMFDB & Newform & LMFDB \\

	\hline

	\endfirsthead

	\caption[]{(continued)}\\

	\hline

	Newform & LMFDB & Newform & LMFDB & Newform & LMFDB & Newform & LMFDB \\

	\hline

	\endhead

	\hline

	\endfoot

	\hline

	\endlastfoot

	$f_{\texttt{24fa1}}$ & \href{https://www.lmfdb.org/ModularForm/GL2/Q/holomorphic/24/2/f/a/11/1/}{\texttt{24\!.\!2\!.\!f\!.\!a\!.\!11\!.\!1}} & $f_{\texttt{153db2}}$ & \href{https://www.lmfdb.org/ModularForm/GL2/Q/holomorphic/153/2/d/b/118/2/}{\texttt{153\!.\!2\!.\!d\!.\!b\!.\!118\!.\!2}} & $f_{\texttt{448fb1}}$ & \href{https://www.lmfdb.org/ModularForm/GL2/Q/holomorphic/448/2/f/b/447/1/}{\texttt{448\!.\!2\!.\!f\!.\!b\!.\!447\!.\!1}} & $f_{\texttt{960oa3}}$ & \href{https://www.lmfdb.org/ModularForm/GL2/Q/holomorphic/960/2/o/a/959/3/}{\texttt{960\!.\!2\!.\!o\!.\!a\!.\!959\!.\!3}} \\

	$f_{\texttt{24fa2}}$ & \href{https://www.lmfdb.org/ModularForm/GL2/Q/holomorphic/24/2/f/a/11/2/}{\texttt{24\!.\!2\!.\!f\!.\!a\!.\!11\!.\!2}} & $f_{\texttt{192ca1}}$ & \href{https://www.lmfdb.org/ModularForm/GL2/Q/holomorphic/192/2/c/a/191/1/}{\texttt{192\!.\!2\!.\!c\!.\!a\!.\!191\!.\!1}} & $f_{\texttt{448fb2}}$ & \href{https://www.lmfdb.org/ModularForm/GL2/Q/holomorphic/448/2/f/b/447/2/}{\texttt{448\!.\!2\!.\!f\!.\!b\!.\!447\!.\!2}} & $f_{\texttt{960oa4}}$ & \href{https://www.lmfdb.org/ModularForm/GL2/Q/holomorphic/960/2/o/a/959/4/}{\texttt{960\!.\!2\!.\!o\!.\!a\!.\!959\!.\!4}} \\

	$f_\texttt{27aa1}$ & \href{https://www.lmfdb.org/ModularForm/GL2/Q/holomorphic/27/2/a/a/1/1/}{\texttt{27\!.\!2\!.\!a\!.\!a\!.\!1\!.\!1}} & $f_{\texttt{192ca2}}$ & \href{https://www.lmfdb.org/ModularForm/GL2/Q/holomorphic/192/2/c/a/191/2/}{\texttt{192\!.\!2\!.\!c\!.\!a\!.\!191\!.\!2}} & $f_{\texttt{504pc1}}$ & \href{https://www.lmfdb.org/ModularForm/GL2/Q/holomorphic/504/2/p/c/307/1/}{\texttt{504\!.\!2\!.\!p\!.\!c\!.\!307\!.\!1}} & $f_{\texttt{1116gc1}}$ & \href{https://www.lmfdb.org/ModularForm/GL2/Q/holomorphic/1116/2/g/c/991/1/}{\texttt{1116\!.\!2\!.\!g\!.\!c\!.\!991\!.\!1}} \\
\end{longtable}
}

{\scriptsize
\begin{longtable}{|c|c|c|c|c|c?c|c|c|c|c|c|}
	\caption{Conjectural identities \eqref{2x2Gamma04} and \eqref{2x2Gamma03} for non-CM quadratic curves}
	\label{conj_for_non-CM_Q-curves}\\

	\hline

	$t$ & $a$ & $b$ & $r$ & $s$ & Type & $t$ & $a$ & $b$ & $r$ & $s$ & Type \\

	\hline

	\endfirsthead

	\caption[]{(continued)}\\

	\hline

	$t$ & $a$ & $b$ & $r$ & $s$ & Type & $t$ & $a$ & $b$ & $r$ & $s$ & Type \\

	\hline

	\endhead

	$8+9 \sqrt{2}$ & $1$ & $-1$ & $98$ & $-\frac{1}{32}$ & \eqref{2x2Gamma04}    &    $10384+2880 \sqrt{13}$ & $1$ & $1$ & $3042$ & $\frac{1}{16}$ & \eqref{2x2Gamma04} \\
	
	\hline

	$8+40 \sqrt{2}$ & $1$ & $1$ & $784$ & $\frac{1}{32}$ & \eqref{2x2Gamma04}    &   $7184+1920 \sqrt{14}$ & $1$ & $-1$ & $2744$ & $-\frac{1}{16}$ & \eqref{2x2Gamma04} \\
	
	\hline

	$1584+1120 \sqrt{2}$ & $1$ & $1$ & $1568$ & $\frac{1}{16}$ & \eqref{2x2Gamma04}   &    $30736+7936 \sqrt{15}$ & $8$ & $-8$ & $37125$ & $-\frac{11}{2}$ & \eqref{2x2Gamma04} \\
	
	\hline

	$9232+6528 \sqrt{2}$ & $1$ & $-1$ & $2304$ & $2$ & \eqref{2x2Gamma04}    &    $8+2 \sqrt{17}$ & $2$ & $2$ & $289$ & $\frac{1}{8}$ & \eqref{2x2Gamma04} \\
	
	\hline

	$112+64 \sqrt{3}$ & $1$ & $-1$ & $864$ & $-\frac{1}{4}$ & \eqref{2x2Gamma04}   &    $8+5 \sqrt{17}$ & $128$ & $128$ & $104329$ & $8$ & \eqref{2x2Gamma04} \\
	
	\hline

	$1552+896 \sqrt{3}$ & $1$ & $-1$ & $1728$ & $4$ & \eqref{2x2Gamma04}     &   $8+8 \sqrt{17}$ & $1$ & $1$ & $578$ & $\frac{1}{16}$ & \eqref{2x2Gamma04} \\
	
	\hline

	$21616+12480 \sqrt{3}$ & $1$ & $-1$ & $2700$ & $\frac{1}{32}$ & \eqref{2x2Gamma04}   &    $18+2 \sqrt{17}$ & $1$ & $1$ & $289$ & $\frac{1}{8}$ & \eqref{2x2Gamma04} \\
	
	\hline

	$301072+173824 \sqrt{3}$ & $2$ & $-2$ & $9261$ & $-\frac{7}{2}$ & \eqref{2x2Gamma04}    &   $\frac{49}{2}+\frac{9 \sqrt{17}}{2}$ & $8$ & $-8$ & $4913$ & $-8$ & \eqref{2x2Gamma04} \\
\hline

\end{longtable}
}

\bibliographystyle{amsplain}
\bibliography{detMMref}

\newcommand{\noop}[1]{}
\providecommand{\bysame}{\leavevmode\hbox to3em{\hrulefill}\thinspace}
\providecommand{\MR}{\relax\ifhmode\unskip\space\fi MR }
\providecommand{\MRhref}[2]{%
  \href{http://www.ams.org/mathscinet-getitem?mr=#1}{#2}
}
\providecommand{\href}[2]{#2}
\begin{thebibliography}{10}

\bibitem{Berndt91}
Bruce~C. Berndt, \emph{Ramanujan's notebooks. {P}art {III}}, Springer-Verlag,
  New York, 1991. \MR{1117903}

\bibitem{Berndt98}
\bysame, \emph{Ramanujan's notebooks. {P}art {V}}, Springer-Verlag, New York,
  1998. \MR{1486573}

\bibitem{BB87}
Jonathan~M. Borwein and Peter~B. Borwein, \emph{Pi and the {AGM}}, Canadian
  Mathematical Society Series of Monographs and Advanced Texts, vol.~4, John
  Wiley \& Sons, Inc., New York, 1998, A study in analytic number theory and
  computational complexity, Reprint of the 1987 original, A Wiley-Interscience
  Publication. \MR{1641658}

\bibitem{Boyd98}
David~W. Boyd, \emph{Mahler's measure and special values of {$L$}-functions},
  Experiment. Math. \textbf{7} (1998), no.~1, 37--82. \MR{1618282}

\bibitem{BZ20}
Fran\c{c}ois Brunault and Wadim Zudilin, \emph{Many variations of {M}ahler
  measures---a lasting symphony}, Australian Mathematical Society Lecture
  Series, vol.~28, Cambridge University Press, Cambridge, 2020. \MR{4382435}

\bibitem{BdLR24}
François Brunault, Rob de~Jeu, Hang Liu, and Fernando~Rodriguez Villegas,
  \emph{${K}_2$ of families of elliptic curves over non-abelian cubic and
  quartic fields}, Preprint (2024),
  \href{https://arxiv.org/abs/2401.04510}{\texttt{arXiv:2401.04510}}.

\bibitem{BLW}
François Brunault, Hang Liu, and Haixu Wang, \emph{Mahler measures of
  $\mathbb{Q}$-curves},  (\noop{2024}in preparation).

\bibitem{Carayol89}
Henri Carayol, \emph{Sur les repr\'{e}sentations galoisiennes modulo {$l$}
  attach\'{e}es aux formes modulaires}, Duke Math. J. \textbf{59} (1989),
  no.~3, 785--801. \MR{1046750}

\bibitem{Cohen93}
Henri Cohen, \emph{A course in computational algebraic number theory}, Graduate
  Texts in Mathematics, vol. 138, Springer-Verlag, Berlin, 1993. \MR{1228206}

\bibitem{CS17}
Henri Cohen and Fredrik Str\"{o}mberg, \emph{Modular forms}, Graduate Studies
  in Mathematics, vol. 179, American Mathematical Society, Providence, RI,
  2017, A classical approach. \MR{3675870}

\bibitem{Cox13}
David~A. Cox, \emph{Primes of the form {$x^2 + ny^2$}}, second ed., Pure and
  Applied Mathematics (Hoboken), John Wiley \& Sons, Inc., Hoboken, NJ, 2013,
  Fermat, class field theory, and complex multiplication. \MR{3236783}

\bibitem{Deninger97}
Christopher Deninger, \emph{Deligne periods of mixed motives, {$K$}-theory and
  the entropy of certain {${\bf Z}^n$}-actions}, J. Amer. Math. Soc.
  \textbf{10} (1997), no.~2, 259--281. \MR{1415320}

\bibitem{DdJZ06}
Tim Dokchitser, Rob de~Jeu, and Don Zagier, \emph{Numerical verification of
  {B}eilinson's conjecture for {$K_2$} of hyperelliptic curves}, Compos. Math.
  \textbf{142} (2006), no.~2, 339--373. \MR{2218899}

\bibitem{Gonzalez11}
Josep Gonz\'{a}lez, \emph{Finiteness of endomorphism algebras of {CM} modular
  abelian varieties}, Rev. Mat. Iberoam. \textbf{27} (2011), no.~3, 733--750.
  \MR{2895332}

\bibitem{GQ14}
Xavier Guitart and Jordi Quer, \emph{Modular abelian varieties over number
  fields}, Canad. J. Math. \textbf{66} (2014), no.~1, 170--196. \MR{3150707}

\bibitem{GJLQ24}
Xuejun Guo, Qingzhong Ji, Hang Liu, and Hourong Qin, \emph{The {M}ahler measure
  of {$x+1/x+y+1/y+4\pm 4\sqrt{2}$} and {B}eilinson's conjecture}, Int. J.
  Number Theory \textbf{20} (2024), no.~1, 185--197. \MR{4688732}

\bibitem{HY22}
Qiao He and Dongxi Ye, \emph{On conjectures of {S}amart}, Manuscripta Math.
  \textbf{167} (2022), no.~3-4, 545--588. \MR{4385383}

\bibitem{KW09-I}
Chandrashekhar Khare and Jean-Pierre Wintenberger, \emph{Serre's modularity
  conjecture. {I}}, Invent. Math. \textbf{178} (2009), no.~3, 485--504.
  \MR{2551763}

\bibitem{KW09-II}
\bysame, \emph{Serre's modularity conjecture. {II}}, Invent. Math. \textbf{178}
  (2009), no.~3, 505--586. \MR{2551764}

\bibitem{LR07}
Matilde~N. Lalin and Mathew~D. Rogers, \emph{Functional equations for {M}ahler
  measures of genus-one curves}, Algebra Number Theory \textbf{1} (2007),
  no.~1, 87--117. \MR{2336636}

\bibitem{LdJ15}
Hang Liu and Rob de~Jeu, \emph{On {$K_2$} of certain families of curves}, Int.
  Math. Res. Not. IMRN (2015), no.~21, 10929--10958. \MR{3456032}

\bibitem{LMFDB}
The {LMFDB Collaboration}, \emph{The {L}-functions and modular forms database},
  \url{https://www.lmfdb.org}, 2024, [Online; accessed 8 September 2024].

\bibitem{MS20}
Yotsanan Meemark and Detchat Samart, \emph{Mahler measures of a family of
  non-tempered polynomials and {B}oyd's conjectures}, Res. Math. Sci.
  \textbf{7} (2020), no.~1, Paper No. 1, 20. \MR{4042306}

\bibitem{Milne72}
J.~S. Milne, \emph{On the arithmetic of abelian varieties}, Invent. Math.
  \textbf{17} (1972), 177--190. \MR{330174}

\bibitem{OEIS_A006203}
{OEIS Foundation Inc.}, \emph{Sequence {A}006203 in {T}he {O}n-{L}ine
  {E}ncyclopedia of {I}nteger {S}equences}, \url{https://oeis.org/A006203},
  Accessed: 2026-05-04.

\bibitem{OEIS_A013658}
\bysame, \emph{Sequence {A}013658 in {T}he {O}n-{L}ine {E}ncyclopedia of
  {I}nteger {S}equences}, \url{https://oeis.org/A013658}, Accessed: 2026-05-04.

\bibitem{Ribet77}
Kenneth~A. Ribet, \emph{Galois representations attached to eigenforms with
  {N}ebentypus}, Modular functions of one variable, {V} ({P}roc. {S}econd
  {I}nternat. {C}onf., {U}niv. {B}onn, {B}onn, 1976), Lecture Notes in Math.,
  Vol. 601, Springer, Berlin-New York, 1977, pp.~17--51. \MR{453647}

\bibitem{Ribet04}
\bysame, \emph{Abelian varieties over {$\bf Q$} and modular forms}, Modular
  curves and abelian varieties, Progr. Math., vol. 224, Birkh\"{a}user, Basel,
  2004, pp.~241--261. \MR{2058653}

\bibitem{Rogers11}
Mathew Rogers, \emph{Hypergeometric formulas for lattice sums and {M}ahler
  measures}, Int. Math. Res. Not. IMRN (2011), no.~17, 4027--4058. \MR{2836402}

\bibitem{RZ12}
Mathew Rogers and Wadim Zudilin, \emph{From {$L$}-series of elliptic curves to
  {M}ahler measures}, Compos. Math. \textbf{148} (2012), no.~2, 385--414.
  \MR{2904192}

\bibitem{Samart15}
Detchat Samart, \emph{Mahler measures as linear combinations of {$L$}-values of
  multiple modular forms}, Canad. J. Math. \textbf{67} (2015), no.~2, 424--449.
  \MR{3314841}

\bibitem{Samart16}
\bysame, \emph{The elliptic trilogarithm and {M}ahler measures of {$K3$}
  surfaces}, Forum Math. \textbf{28} (2016), no.~3, 405--423. \MR{3510822}

\bibitem{Samart21}
\bysame, \emph{A functional identity for {M}ahler measures of non-tempered
  polynomials}, Integral Transforms Spec. Funct. \textbf{32} (2021), no.~1,
  78--89. \MR{4238025}

\bibitem{Samart23}
\bysame, \emph{Mahler measure of a non-reciprocal family of elliptic curves},
  Q. J. Math. \textbf{74} (2023), no.~3, 1187--1208. \MR{4642252}

\bibitem{ST24}
Detchat Samart and Zhengyu Tao, \emph{Supplementary data for {D}eterminants of
  {M}ahler measures and special values of {$L$}-functions}, Available at
  \url{https://github.com/petesamart/Data-file-for-det-of-MM/blob/main/data.pdf},
  2024.

\bibitem{Schoeneberg74}
Bruno Schoeneberg, \emph{Elliptic modular functions: an introduction}, Die
  Grundlehren der mathematischen Wissenschaften, Band 203, Springer-Verlag, New
  York-Heidelberg, 1974, Translated from the German by J. R. Smart and E. A.
  Schwandt. \MR{412107}

\bibitem{Stein07}
William Stein, \emph{Modular forms, a computational approach}, Graduate Studies
  in Mathematics, vol.~79, American Mathematical Society, Providence, RI, 2007,
  With an appendix by Paul E. Gunnells. \MR{2289048}

\bibitem{Sturm87}
Jacob Sturm, \emph{On the congruence of modular forms}, Number theory ({N}ew
  {Y}ork, 1984--1985), Lecture Notes in Math., vol. 1240, Springer, Berlin,
  1987, pp.~275--280. \MR{894516}

\bibitem{TG25}
Zhengyu Tao and Xuejun Guo, \emph{C{M} points, class numbers, and the {M}ahler
  measures of {$x^3+y^3+1-kxy$}}, Math. Comp. \textbf{94} (2025), no.~351,
  425--446. \MR{4807816}

\bibitem{TGW24}
Zhengyu Tao, Xuejun Guo, and Tao Wei, \emph{Mahler measures and {$L$}-values of
  elliptic curves over real quadratic fields}, Preprint (2024),
  \href{https://arxiv.org/abs/2209.14717}{\texttt{arXiv:2209.14717}}.

\bibitem{Taylor02}
Richard Taylor, \emph{Remarks on a conjecture of {F}ontaine and {M}azur}, J.
  Inst. Math. Jussieu \textbf{1} (2002), no.~1, 125--143. \MR{1954941}

\bibitem{Lambda}
{T}he {M}athematical~{F}unctions {G}rimoire, \emph{Modular lambda function},
  \url{https://fungrim.org/topic/Modular_lambda_function/}, Accessed:
  2024-08-19.

\bibitem{PARI2}
{The PARI~Group}, Univ. Bordeaux, \emph{{PARI/GP version \texttt{2.15.5}}},
  2024, available from \url{http://pari.math.u-bordeaux.fr/}.

\bibitem{vH15}
Mark van Hoeij, \emph{Parametrization of the modular curve {$X_0(N)$}},
  \url{https://www.math.fsu.edu/~hoeij/files/X0N/Parametrization}, Accessed:
  2024-08-19.

\bibitem{RV99}
F.~Rodriguez Villegas, \emph{Modular {M}ahler measures. {I}}, Topics in number
  theory ({U}niversity {P}ark, {PA}, 1997), Math. Appl., vol. 467, Kluwer Acad.
  Publ., Dordrecht, 1999, pp.~17--48. \MR{1691309}

\bibitem{YZ97}
Noriko Yui and Don Zagier, \emph{On the singular values of {W}eber modular
  functions}, Math. Comp. \textbf{66} (1997), no.~220, 1645--1662. \MR{1415803}

\end{thebibliography}

\end{document}